\documentclass{amsart}

\usepackage{amssymb}
\usepackage{mathtools}
\usepackage[utf8]{inputenc}
\usepackage[T1]{fontenc}
\usepackage[english]{babel}
\usepackage{xargs}
\usepackage{graphicx}
\usepackage[svgnames]{xcolor}
\usepackage[figurewithin=none]{caption,subcaption} \usepackage{rotating}\usepackage{enumitem}
\usepackage{array, multirow}

\usepackage{tabularx}
\newcolumntype{\ll}[1]{>{\hsize=#1\hsize\raggedright\arraybackslash}X}\newcolumntype{\rr}[1]{>{\hsize=#1\hsize\raggedleft\arraybackslash}X}\newcolumntype{\cc}[1]{>{\hsize=#1\hsize\centering\arraybackslash}X}

\usepackage[style=numeric, sorting=nyt]{biblatex}
\usepackage{hyperref}
\hypersetup{
	hidelinks,
  colorlinks  = true,    urlcolor    = blue,    linkcolor   = blue,    citecolor   = blue,      pdfauthor   = {A. Montero and A. Weiss},pdfsubject  = {Locally spherical regular hypertopes},pdftitle    = {Proper locally spherical hypertopes of hyperbolic type},  }

\usepackage[capitalise,noabbrev, nameinlink]{cleveref}
\usepackage{tikz-cd}
\tikzcdset{start anchor = center, end anchor = center, shorten=1mm}

\usepackage[shadow, colorinlistoftodos,textsize=tiny]{todonotes}
\setuptodonotes{fancyline,backgroundcolor=gray!10, bordercolor=blue}

\newcommandx{\inline}[2][1=]{\todo[inline,#1]{#2}}

\makeatletter
    \providecommand\@dotsep{5}
\makeatother

\theoremstyle{plain}
\newtheorem{thm}{Theorem}[section]
\newtheorem{theorem}[thm]{Theorem}

\newtheorem{lemma}[thm]{Lemma}

\newtheorem{proposition}[thm]{Proposition}

\newtheorem{corollary}[thm]{Corollary}

\theoremstyle{definition}

\theoremstyle{remark}

\newtheorem{remark}[thm]{Remark}

\numberwithin{equation}{section}

\usepackage{xargs}

\newcommand{\bN}{\mathbb{N}}
\newcommand{\bZ}{\mathbb{Z}}

\newcommand{\cH}{\mathcal{H}}

\newcommand{\cP}{\mathcal{P}}

\newcommand{\cX}{\mathcal{X}}
\newcommand{\cY}{\mathcal{Y}}

\newcommand{\rr}{\varrho}
\newcommand{\id}{\varepsilon}

\DeclareMathOperator{\aut}{Aut} \DeclareMathOperator{\autI}{Aut_{I}} \DeclareMathOperator{\stab}{Stab}

\renewcommand{\leq}{\leqslant} \renewcommand{\geq}{\geqslant}
\renewcommand{\epsilon}{\varepsilon} \renewcommand{\subset}{\subseteq}  
\newcommand{\sm}{\setminus}

\DeclareMathOperator{\lcm}{lcm}

\begin{document}

\title{Proper locally spherical hypertopes of hyperbolic type}
 \thanks{This research was supported by NSERC Canada and DGAPA-UNAM Mexico}

\author[A. Montero]{Antonio Montero}
\address[A. Montero]{Institute of Mathematics, National Autonomous University of Mexico (IM UNAM), 04510 Mexico City, Mexico}
\email[A. Montero]{amontero@im.unam.mx}

\author[A.I. Weiss]{Asia Ivi\'c Weiss}
\address[A.I. Weiss]{Department of Mathematics and Statistics, York University, Toronto, Ontario M3J 1P3, Canada}
\email[A.I. Weiss]{weiss@yorku.ca}



\begin{abstract}
Given any irreducible Coxeter group $C$ of hyperbolic type with non-linear diagram and rank at least $4$, whose maximal parabolic subgroups are finite, we construct an infinte family of locally spherical regular hypertopes of hyperbolic type whose Coxeter diagram is the same as that of $C$.
\end{abstract}

\maketitle

\section{Historical overview and motivation}

This paper deals with combinatorial objects possessing large degree of symmetry whose structure has been motivated by the geometry of the polyhedra and, their higher dimensional analogues, the polytopes. 
Highly symmetric polyhedra have occupied mathematicians and non-mathematicians alike for a very long time. 
Their study has contributed to several areas of mathematics and have been used in other sciences (notably physics and chemistry). 
These beautiful structures also motivated numerous creations in applied and fine arts. 
Some of the five classical regular polyhedra, known as Platonic solids or regular convex polyhedra, have been known to late Neolithic people as along as 3,000 years ago. 
The first proof that there are only five of these objects, most likely due to Theaetetus, appears in the Book XIII of the Euclid's \emph{Elements} \cite{Euclid_2002_EuclidsElements}.

The five regular solids have non-intersecting congruent regular convex polygons as faces with the same arrangement of faces at each vertex. 
They can be seen as maps: cellular decompositions of a surface where each face is represented by a cellular region and where a number of regions can meet at a common corner, the vertex of the map. 
The regular polyhedra are globally spherical (meaning that their faces form a cellular decomposition of a sphere) and have highest possible symmetry. 
A polyhedron, with $p$-gonal faces and $q$-gonal vertex arrangements (and the associated map) is usually denoted by the Schläfli symbol $\{p,q\}$, called the \emph{type} of the polyhedron (map).
The five polyhedra are the tetrahedron of type $\{3,3\}$, the octahedron of type $\{3,4\}$, the cube of type $\{4,3\}$, the icosahedron of type $\{3,5\}$, and the dodecahedron of type $\{5,3\}$. 
Between 1850 and 1852 Schläfli extends the concept of a polyhedron to a higher dimensional object that was eventually named a polytope. 
His work was unappreciated at that time and only published posthumously in 1901 (see, for example \cite{Coxeter_1973_RegularPolytopes} by Coxeter). 
In the introduction to Chapter VII of \cite{Coxeter_1973_RegularPolytopes} Coxeter writes ``The historical remarks in $\S$ 7.x  are dominated by the name of one man, Schläfli, to whom practically all developments are due''.

Already in the 17th century other kinds of regular polyhedra were considered. 
The faces of these polyhedra could be non-convex congruent regular star polygons, or the faces (while still required to have the same arrangement at each vertex) could overlap. 
There are precisely four such objects as proved by Cauchy in 1912, and they are known as the Kepler-Poinsot polyhedra or star polyhedra. 
Combinatorially, the star polyhedron of type $\{5,5/2\}$ and its dual of type $\{5/2,5\}$ (here the symbol $5/2$ denotes a pentagram) are maps on a surface of genus 4. 
The remaining two star polyhedra $\{3,5/2\}$ and its dual $\{5/2,3\}$ are combinatorially equivalent respectively to an icosahedron and a dodecahedron and therefore are of spherical type. 

In 1937, in another direction, Coxeter also considered polyhedra with regular skew polygonal faces \cite{Coxeter_1937_RegularSkewPolyhedra}. 
In 1977 Grünbaum extends the definition of a regular polyhedron further to include polyhedra whose faces can be any kind of congruent regular polygons (see \cite{Gruenbaum_1977_RegularPolyhedraOld}).
In \cite{Gruenbaum_1994_PolyhedraHollowFaces} Grünbaum writes ``The Original Sin in the theory of polyhedra goes back to Euclid, and through Kepler, Poinsot, Cauchy, and many others continues to afflict all work on this topic''.  
Schläfli's ideas and the early work by  Grünbaum and Coxeter in extending the definition of a regular polytope was developed further by Danzer and Schulte \cite{DanzerSchulte_1982_RegulareInzidenzkomplexe.I}. 
This eventually resulted in a combinatorial concept of a regular abstract polytope (see \cite{McMullenSchulte_2002_AbstractRegularPolytopes}). 
The connection with the classical theory is very strong as the faces of an abstract polytope are required to form a partially ordered set. 
In \cite[Section 2E]{McMullenSchulte_2002_AbstractRegularPolytopes} it is proved that the automorphism groups of regular abstract polytopes are quotients of Coxeter groups with linear diagrams satisfying an intersection condition on its parabolic subgroups, and that every such group can be used to construct an abstract regular polytope. 

While each classical regular polytope is topologically a sphere, abstract regular polytopes can be of other topological types. Spherical regular abstract polytopes are isomorphic to classical regular polytopes and are associated with finite irreducible Coxeter groups with linear diagrams.
Toroidal abstract regular polytopes are associated with finite quotients of affine irreducible Coxeter groups with linear diagrams. 
These have been classified in rank 3 by Coxeter (see \cite[Sections 8.3 and 8.4]{CoxeterMoser_1972_GeneratorsRelationsDiscrete}) and in higher ranks by McMullen and Schulte (see \cite[Sections 6D and 6E]{McMullenSchulte_2002_AbstractRegularPolytopes}).

Motivated by the above ideas, in 2016 the concept of a regular polytope is extended further by Leemans, Fernandes, and the second author of this paper, to that of a regular hypertope. 
In \cite{FernandesLeemansWeiss_2016_HighlySymmetricHypertopes} they pushed the generalisation as far as it can possibly be done, while ensuring that the new objects locally have the natural structure essentially resembling that of a polytope. 
The essential difference is that it is not required that the combinatorial structure of a hypertope is that of a partially ordered set. 
More precisely, regular hypertopes are thin, residually connected incidence geometries with chamber transitive type-preserving automorphism groups. 
The abandonment of the partial order requires that we can not consider all automorphisms of the incidence geometry but must restrict to those preserving the types of the elements of the geometry. 
For details see \cref{sec:basics}. 
In \cite{FernandesLeemansWeiss_2016_HighlySymmetricHypertopes} it is proved that such an object can be constructed from a quotient of a Coxeter group endowed with certain combinatorial and geometric properties (see \cref{thm:CGroupFlagTrans_Hypertope} in this paper). 
When the Coxeter group has a linear diagram the resulting hypertope is a polytope, otherwise we say that the hypertope is proper.

Of particular interest again are the spherical and toroidal regular hypertopes that are associated with the finite and affine irreducible Coxeter groups, respectively. 
The proper toroidal hypertopes (with non-linear diagram) in rank $4$ have been classified by Ens in \cite{Ens_2018_Rank4Toroidal} and in higher ranks by Leemans, Weiss, and Schulte in a forthcoming paper \cite{LeemansSchulteWeiss_ToroidalHypertopes_preprint}. Both spherical and toroidal regular hypertopes  have spherical residues, meaning that if one or more of the generators of their type-preserving automorphism group are dropped the resulting hypertope of lower rank is spherical.

In \cite{FernandesLeemansWeiss_2020_ExplorationLocallySpherical} the concept of a locally spherical hypertope was introduced as a thin incidence geometry for which all its proper residues are spherical. 
The type-preserving automorphism group of a locally spherical regular hypertope is then seen to be a quotient of a finite irreducible, affine irreducible, or a compact hyperbolic Coxeter group. 
The classification of locally spherical hypertopes of spherical and euclidian type follows from the classification of the normal subgroups of finite and infinite irreducible Coxeter groups of euclidean type, respectively. 
Little is known about the remaining locally spherical regular hypertopes. 
They are obtained from quotients of compact hyperbolic Coxeter groups. 
A review of the know finite hypertopes (most of which are polytopes) of this type is presented in the paper and an attempt was made to construct finite (small) examples. 
The calculations in MAGMA led to very few examples because of the large size of the groups involved. 
The authors were unsuccessful in finding any finite examples for proper hypertopes in ranks $4$ and $5$. 
One example in rank $5$ with a Y-shape Coxeter diagram was subsequently found \cite{MonteroWeiss_2020_LocallySphericalHypertopes}. 

In \cref{sec:constructions} of this paper we give examples for each of the compact hyperbolic Coxeter groups with non-linear diagram. 
In fact, for each such group we construct an infinite family of hypertopes.
Our constructions extend the results in \cite[Section 4C5]{McMullenSchulte_2002_AbstractRegularPolytopes} on existence of infinitely many regular polytopes of hyperbolic type in ranks $4$ and $5$ to hypertopes. 
Furthermore, our results are constructive providing explicitly the type-preserving automorphisms groups of such hypertopes.
The constructions in ranks $4$ and $5$, described in \cref{sec:constructions}, are based on Schreier coset graphs (for a brief description of the graphs and the computational methods for handling them see \cite{Conder_2003_GroupActionsGraphs}) which are useful in forming large permutation representations of the groups. These in turn are related to the systematic enumeration of cosets derived by Coxeter and Todd  \cite{ToddCoxeter_1936_PracticalMethodEnumerating}. 
 \section{Construction of regular hypertopes from groups}\label{sec:basics}
Hypertopes are combinatorial objects which in a natural way extend both concepts of polytopes and hypermaps. 
They were introduced in \cite{FernandesLeemansWeiss_2016_HighlySymmetricHypertopes}  within the theory of incidence geometries.
Formally speaking, a \emph{hypertope} $\cH$ is an incidence system $(X, \ast, t, I)$ (where $X$ is the set of elements, $\ast$ is the incidence relation, $t$ is the type function and $I$ is the set of types), which is a thin, residually connected geometry.
We will assume familiarity with these concepts and direct the readers to \cite{BuekenhoutCohen_2013_DiagramGeometry} for definitions as well as basic examples of geometries. 

A \emph{flag} of $\cH$ is a set of mutually incident elements of the geometry. 
The \emph{rank} of a flag $F$ is $|F|$ and the type of $F$ is $\left\{ t(x) : x \in F \right\} \subset I $.
A maximal flag is called a \emph{chamber}. 
Since $\cH$ is a geometry, all its chambers have rank $|I|$.
Usually the types are denoted by non-negative integers so that $I=\{0, \dots, n-1\}$ in which case we say that $\cH$ is of rank $n$.
Observe that hypertopes of rank $2$ are essentially (abstract) polygons and those of rank $3$ are non-degenerate hypermaps (including maps as a special case).

An automorphism $\gamma:X\to X$ of $\cH$ is said to be \emph{ type-preserving} if $t(x \gamma) = t(x)$ for every $x \in X$.
The group of type-preserving automorphisms of $\cH$ is denoted by $\autI(\cH)$.
This group acts naturally on the set of chambers of $\cH$. 
A direct consequence of residual connectivity is that this action is free. 
A hypertope  $\cH$ is said to be \emph{regular} if the group $\autI(\cH)$ acts transitively on the set of its chambers. 
As a consequence of the thinness of $\cH$, for every chamber $C$ and every $i \in I$ there exists a unique chamber $C^{i}$ that differs form $C$ only in the element of type $i$. 
It follows that if $\cH$ is regular and $C$ is one of its chambers, for every $i \in I$ there exists a unique type-preserving automorphism $\rho_{i}$ that maps $C$ to $C^{i}$.

It is easy to see that $\autI(\cH)=\left\langle \rho_{i} : i \in I \right\rangle $ (see \cite[Section 4]{FernandesLeemansWeiss_2016_HighlySymmetricHypertopes}). 
The automorphisms $\rho_{0}, \dots, \rho_{n-1}$ satisfy the relations
\begin{equation}\label{eq:coxeter_rels}
    \begin{aligned}
    \rho_{i}^{2} &= \id && \text{for every $i \in I$},\\
    (\rho_{i}\rho_{j})^{p_{i,j}} &= \id && \text{for some } p_{i,j}\in \left\{ 2, 3, \dots, \infty \right\}. 
    \end{aligned}
\end{equation}

Let $G$ be a group generated by the elements $\rho_{0}, \dots, \rho_{n-1}$ that satisfy \cref{eq:coxeter_rels}. 
The \emph{Coxeter diagram} of $G$ is the complete graph with nodes labelled by the elements $\rho_{0}, \dots, \rho_{n-1}$ and the branch connecting the node $\rho_{i}$ with the node $\rho_{j}$ is labelled by $p_{i,j}$.
As usual, branches are omitted when $p_{i,j}=2$ and labels are omitted if $p_{i,j}=3$.

A regular hypertope is said to be \emph{irreducible} if the Coxeter diagram of its type-preserving automorphism group is connected (after removing the branches labelled by $2$).

The group defined only by the relations in \cref{eq:coxeter_rels} is called the \emph{(universal) Coxeter group} associated with the corresponding diagram.
Observe that the type-preserving automorphism group of a regular hypertope is a quotient of the Coxeter group associated with its diagram.
However, the converse is not always true, meaning that not every quotient of a Coxeter group is the type-preserving automorphism group of a regular hypertope. 
In general, we need additional conditions, which are discussed below.

Let $G$ be a group generated by the involutions $\left\{ \rho_{i} : i \in I \right\} $. 
For $J \subset I$, define $G_{J}=\left\langle \rho_{j} : j \in J \right\rangle $.
We say that $G$ is a \emph{C-group} if it satisfies that 
\begin{equation} \label{eq:IP}
    G_{J} \cap G_{K} = G_{J \cap K},
\end{equation}
for every $J,K \subset I$. This condition is called the \emph{intersection property}.

In general, proving that a given group satisfies the intersection property can be difficult. For $i,j \in I$, consider the subgroups $G_{i}= G_{I \sm \left\{ i \right\} }$ and $G_{i,j} = G_{I\sm \left\{ i,j \right\}  }$.
The following result offers a way to prove the intersection property in terms of these subgroups.

\begin{proposition}[{\cite[Proposition 6.1]{FernandesLeemans_2018_CGroupsHigh}}]\label{prop:IP} 	
    Let $G$ be a group generated by $n$ involutions  $\rho_0,\ldots, \rho_{n-1}$. 
    Suppose that $G_i$ is a C-group for every $i\in\{0,\ldots,n-1\}$. 
    Then $G$ is a C-group if and only if $G_i \cap G_j = G_{i,j}$ for all $0\leq i,\,j \leq n-1$.
\end{proposition}

It is known that the type-preserving automorphism group of a regular hypertope is a C-group (see \cite[Theorem 4.1]{FernandesLeemansWeiss_2016_HighlySymmetricHypertopes}). 
Every Coxeter group satisfies the intersection property (see \cite[Section 1.13]{Humphreys_1990_ReflectionGroupsCoxeter}).
In the special case where the hypertope is a polytope, the Coxeter diagram is linear and the intersection property is sufficient, meaning that for every C-group $G$ with linear diagram there exists a regular polytope $\cP$ such that $\aut(\cP)=G$ (see \cite[Theorem 2E11]{McMullenSchulte_2002_AbstractRegularPolytopes}).

We say that a regular hypertope is a \emph{proper} hypertope if it is not isomorphic to an abstract polytope; equivalently if its Coxeter diagram is not linear.
If $G$ is a group generated by involutions with a non-linear Coxeter diagram, the intersection property is not sufficient to guarantee the existence of a regular hypertope with $G$ as its type-preserving automorphism group (see \cite[Example 3.3]{FernandesLeemansWeiss_2020_ExplorationLocallySpherical}).  
Next we discuss a condition called {flag-transitivity}.
In \cite{FernandesLeemansWeiss_2016_HighlySymmetricHypertopes} it is shown that if this condition is satisfied by a C-group, the following construction introduced by Tits in \cite{Tits_2013_GroupesEtGeometries} leads to a regular hypertope.

Let $n$ be a positive integer, $I:=\{0,\ldots,n-1\}$ and $G$ a group generated by the involutions $\left\{ \rho_{i}: i \in I \right\}$.
For $i \in I$, let $G_{i}$ be the subgroup $\left\langle \rho_{j} : j \in I \sm \left\{ i \right\}  \right\rangle $, $X$  the set of cosets $G_ig$ with $g\in G$ and $i\in I$, and $t:X \to I$ defined by $t(G_ig)=i$. 
Define an incidence relation $\ast$ on $X\times X$ by: 
\[\begin{aligned}
G_ig_1 &\ast G_jg_2 &&\text{ if and only if }& G_ig_1\cap G_jg_2 &\neq \emptyset .
\end{aligned}\]
Then the $4$-tuple $\Gamma:=(X,\ast, t,I)$ is an incidence system having $\{G_{i} : i \in I\}$ as a chamber.
Moreover, the group $G$ acts by right multiplication as an automorphism group on $\Gamma$ and it is transitive on the flags of rank at most $2$.
An incidence system built this way is denoted by $\Gamma\left( G, \left( G_{i} \right)_{i \in I} \right)$.

If $G$ is the type-preserving automorphism group of a regular hypertope $\cH$, then $\Gamma\left( G, \left( G_{i} \right)_{i \in I} \right) \cong \cH$. 
Note that since $\cH$ is regular, $G$ acts transitively on the set of flags of $\cH$ with the same type and we say that $G$ is \emph{flag-transitive}. 
This property is sufficient to characterise the type-preserving automorphism groups of the regular hypertopes. 
More precisely, we have the following result.

\begin{theorem}[{\cite[Theorem 4.6]{FernandesLeemansWeiss_2016_HighlySymmetricHypertopes}}]\label{thm:CGroupFlagTrans_Hypertope}
    Let $I= \{0, \dots, n-1\}$, let $G=\langle \rho_i\,|\,i\in I\rangle$ be a C-group, and let $\Gamma := \Gamma(G;(G_i)_{i\in I})$ where $G_i:=\langle \rho_j\,|\,j\neq i\rangle$ for all $i\in I$.
    If $G$ is flag-transitive on $\Gamma$, then $\Gamma$ is a regular hypertope.
\end{theorem}

It is known that every Coxeter group is flag transitive, hence the type-preserving automorphism group of a regular hypertope (see \cite[Section 3]{Tits_2013_GroupesEtGeometries}). Hence, with each Coxeter diagram we can associate the \emph{universal regular hypertope} built from the corresponding Coxeter group.

The following two results will be used to determine whether a given C-group is flag-transitive.

\begin{theorem}[{\cite[Theorem 1.8.10 $(iii)$ ]{BuekenhoutCohen_2013_DiagramGeometry}}] \label{thm:FT_rank3}
    Let $n \geq 4$ and $I=\{0, \dots, n-1\}$. 
    Let $\Gamma=\Gamma(G,G_{i})_{i \in I}$ be the incidence system of $G$ over $(G_{i})_{i \in I}$. 
    Then $\Gamma$ is flag-transitive if and only if for each subset $J$ of $I$ of size three, the group $G$ is transitive on the set of flags of type $J$, and for each $i \in I$ the subgroup $G_{i}$ is flag-transitive on $\Gamma(G_{i},G_{i,j})_{j \in I\sm\{i\}}$. 
\end{theorem}

\begin{lemma}[{\cite[Lemma 4.2]{FernandesLeemansWeiss_2016_HighlySymmetricHypertopes}}]\label{lem:FT_Group}
    Let $H$, $K$, $Q$ be three subgroups of a group $G$. 
    Then the following conditions are equivalent.
    \begin{enumerate}
    \item\label{item:FT_Group1} $Q\cdot H \cap Q\cdot K =Q \cdot (H\cap K)$
    \item\label{item:FT_Group2} $(Q\cap H) \cdot (Q \cap K) =Q\cap ( H\cdot K)$
    \item\label{item:FT_Group3} If the three cosets $Qx$, $Hy$ and $Kz$ have pairwise nonempty intersections, then $Qx \cap Hy \cap Kz \neq \emptyset$.
    \end{enumerate}
\end{lemma}

 \section{Locally spherical regular hypertopes of hyperbolic type}\label{sec:lsrhht}
Let $\cH$ be a regular hypertope of rank $n$, $I=\left\{ 0, \dots, n-1 \right\} $ and let $G$ be its type-preserving automorphism group.
Following the notation in \cref{sec:basics}, when $G=\left\langle \rho_{i} : i \in I \right\rangle $, for every $i,j \in I$ we define $G_{i} = \left\langle \rho_{j} : j \in I \sm \left\{ i \right\}  \right\rangle $ and $G_{i,j} = \left\langle \rho_{k} : k \in I \sm \left\{ i,j \right\}  \right\rangle $. 
The \emph{maximal residue of $\cH$ of type $i$} is the incidence system $\Gamma\left( G_{i}, \left( G_{i,j} \right)_{j \in I \sm \left\{ i \right\} } \right)$. 
The group $G_{i}$ is a flag-transitive C-group, hence all the maximal residues of a regular hypertope are regular hypertopes themselves.
Observe that for regular hypertopes this definition is equivalent to the one introduced in \cite{FernandesLeemansWeiss_2016_HighlySymmetricHypertopes}.

A universal regular hypertope is \emph{spherical} if its Coxeter diagram is a union of diagrams of finite irreducible Coxeter groups. 
A \emph{locally spherical regular hypertope} is a hypertope whose maximal residues are spherical hypertopes.
An irreducible locally spherical regular hypertope is of \emph{hyperbolic type} if its Coxeter diagram is the same as that of an irreducible compact hyperbolic Coxeter group, that is,  a group generated by hyperbolic reflections with a compact fundamental domain.
Compact hyperbolic Coxeter groups exist only in ranks $3$, $4$ and $5$ (see \cite[Section 6.9]{Humphreys_1990_ReflectionGroupsCoxeter}) and they are listed in \cite[Table 2]{FernandesLeemansWeiss_2020_ExplorationLocallySpherical}.
In this paper we only consider ranks $4$ and $5$ since, as mentioned before, rank $3$ hypertopes are non-degenerate hypermaps for which there are several known constructions (see \cite{Conder_2009_RegularMapsHypermaps,CornSingerman_1988_RegularHypermaps,JonesSingerman_1994_MapsHypermapsTriangle}, for example).

When the Coxeter diagram associated with a hypermap is a triangular graph with edges labelled $l, m, n$ the hypermap is said to be of \emph{type} $(l,m,n)$.
In \cite{FernandesLeemansWeiss_2020_ExplorationLocallySpherical} this notation was extended to hypertopes of ranks $4$ and $5$. 
We follow that notation and say that a hypertope of rank $n$ is of type $(p_{1}, \dots, p_{n})$ when its Coxeter diagram is a cycle with successive edges labelled by $p_{1}, \dots, p_{n}$.
Hypertopes described in \cref{sec:5-33,sec:53-33} are naturally associated with semiregualr polytopes (see \cite{Coxeter_1973_RegularPolytopes}).
We denote the type of such hypertopes using Coxeter's notation.

A survey of known locally spherical regular polytopes of hyperbolic type can be found in \cite[Section 6]{FernandesLeemansWeiss_2020_ExplorationLocallySpherical}. 
In the same paper a computational attempt to construct examples of proper hypertopes of such type was made.
Just few examples were found as the groups involved seemed to be large.
The results are summarised in \cite[Table 3]{FernandesLeemansWeiss_2020_ExplorationLocallySpherical}.
Notably, the authors of that work could not find a finite example of a proper locally spherical regular hypertope of rank $5$.
Recently the authors of this manuscript found one finite such example in \cite{MonteroWeiss_2020_LocallySphericalHypertopes}
with large type-preserving automorphism group. 
In  this paper we use CPR-graphs (discussed in next section) to build infinite families of locally spherical regular hypertopes for each Coxeter diagram of hyperbolic type of ranks $4$ and $5$.
Most of the groups involved are fairly large.

 \section{CPR-Graphs} \label{sec:cpr_graphs}
In this section we introduce the notion of a CPR-graph, which will be our main tool for building the type-preserving automorphism groups of regular hypertopes. 
The term “CPR-graph” stands for (string) C-group permutation representation of a graph. 
The graphs, introduced by Pellicer in \cite{Pellicer_2008_CprGraphsRegular}, were used to build regular polytopes with symmetric and alternating groups.
CPR-graphs have proved to be a powerful tool in building highly symmetric polytopes. 
In \cite{Pellicer_2009_ExtensionsRegularPolytopes, Pellicer_2010_ExtensionsDuallyBipartite} the graphs were used to build regular polytopes with prescribed Schläfli symbol. 
In \cite{PellicerWeiss_2010_GeneralizedCprGraphs} Pellicer and Weiss generalise CPR-graphs to GPR-graphs (generalised permutation representation graphs) and use this generalisation to build chiral polytopes of small rank. 
GPR-graphs were also used in \cite{Pellicer_2010_ConstructionHigherRank} to show the   existence of chiral polytopes of arbitrary rank.
In \cite{CunninghamPellicer_2014_ChiralExtensionsChiral} Cunningham and Pellicer used GPR-graphs to construct chiral polytopes with prescribed chiral facets.
Recently Pellicer, Toledo and Poto\v{c}nik used GPR-graphs to build $2$-orbit maniplexes for every rank and every symmetry type. 
In another direction in \cite{FernandesPiedade_2019_FaithfulPermutationRepresentations} Fernandes and Piedade classify the CPR-graphs of regular maps on tori and in \cite{FernandesPiedade_2020_DegreesToroidalRegular_preprint} they extend those results to regular toroidal hypermaps.

CPR-graphs admit a natural generalisation for C-groups with non-linear Coxeter diagram.

Let $G$ be a group generated by involutions $\rho_{0}, \rho_{1} \dots, \rho_{n-1}$.
Assume that $\pi: G \to S_{m}$ is an embedding of $G$ into a symmetric group $S_{m}$ for a certain $m \in \bN$.
The \emph{CPR-graph} associated with $\pi$ is the edge-coloured graph whose vertex set is $\left\{ 1, \dots, m \right\} $ and such that there is an edge of colour $i$ (for $i \in \left\{ 0, \dots, n-1 \right\} $) connecting the vertices $x$ and $y$ if and only if $x \pi(\rho_{i}) = y$.  

Usually, the embedding $\pi$ is given by a known faithful action of the group $G$ on a certain set with $m$ elements and can be omitted.

Observe that the action of $G$ on $\left\{ 1, \dots, m \right\} $ is transitive if and only if the corresponding CPR-graph $\cX$ is connected. 
In this situation, the stabiliser $S \leq G$ of a vertex has index $m$ and the graph $\cX$ is isomorphic to the Schreier coset graph induced by $S$ (see \cite{Schreier_1927_DieUntergruppenDer, ToddCoxeter_1936_PracticalMethodEnumerating} and \cite[Proposition 3.10]{Pellicer_2008_CprGraphsRegular}).
However, in general CPR-graphs do not need to be connected.

Observe that if $\cX$ is a CPR-graph of $G$, then for every $i \in \left\{ 0, \dots, n-1 \right\} $, the edges of colour $i$ form a matching $M_{i}$.
We can recover $G$ from $\cX$ as a permutation group by defining $\rho_{i}$ as the involution given by swapping the endpoints of the edges of the matching $M_{i}$.

Inspired by the previous observation and following \cite{Pellicer_2008_CprGraphsRegular}, we say that a (multi) graph $\cX$ is a \emph{ proper $n$-edge-coloured graph} if the edges of colour $i$ form a non-empty matching $M_{i}$ for every $i \in \left\{ 0, \dots n-1 \right\} $ and such that if $M_{i} \neq M_{j}$ if $i \neq j$.

Observe that every proper $n$-edge-coloured graph defines a permutation group $G$ generated by the involutions given by swapping the endpoints of the edges of each matching $M_{i}$.
In \cref{sec:constructions} we build an infinite family of proper $n$-edge-coloured graphs for every non-linear Coxeter diagram of hyperbolic type. 
Then we prove that the induced permutation groups are the type-preserving automorphism groups of regular hypertopes, obtaining as a consequence an infinite family of proper locally spherical regular hypertopes for each hyperbolic type.

To finish this section we prove some results that will be useful in showing that the induced permutation group satisfy the required relations as well as the intersection property and flag transitivity condition required to be the type-preserving automorphism group of a regular hypertope.

The following remark is obvious. This is essentially \cite[Lemma 5.3]{MonsonPellicerWilliams_2014_MixingMonodromyAbstract} in the language of CPR-graphs.

\begin{remark}\label{rem:mix}
	Let $G$ and $H$ be the permutation groups induced by the proper $n$-edge-coloured graphs $\cX$ and $\cY$, respectively. 
	Assume that $H$ is a quotient of $G$ mapping distinguished generators to distinguished generators.
	Then the permutation group induced by the disjoint union of $\cX$ and $\cY$ is isomorphic to $G$.
\end{remark}

Let $\cX$ be a proper $n$-edge-colured graph with colour set $I=\left\{ 0, \dots, n-1 \right\} $.
If $J \subset I $, a \emph{$J$-component} is a connected component of the subgraph of $\cX$ induced by the edges of colours in $J$. 
Note that if $\left\{ i,j \right\} \subset I $, then every $\left\{ i,j \right\} $-component is either an alternating path or an alternating cycle.

\begin{lemma}\label{lem:polygonalAction}
    Let $\cX$ be a proper $n$-edge-coloured graph with colour-set $I = \left\{ 0, \dots, n-1 \right\} $. 
    For $i \in I $, let $\rho_{i}$ denote the permutation of the vertices of $\cX$ induced by the matching $M_{i}$. 
    Let $p_{1}, \cdots p_{r}, q_{1}, \dots, q_{s} \in \bN$ such that all the $\left\{ i,j \right\} $-components of $\cX$ are alternating paths with $p_{k}$ or alternating cycles with $2q_{l}$ vertices for some $1\leq k \leq r$ and $1 \leq l \leq s$. Then, the period of $\rho_{i} \rho_{j}$ is $\lcm\left( p_{1}, \cdots p_{r}, q_{1}, \dots, q_{s} \right)$.
\end{lemma}

\begin{lemma}\label{lem:IP_CPR}
    Let $\cX$ be a proper $n$-edge-coloured graph with colour-set $I = \left\{ 0, \dots, n-1 \right\} $. 
    For $i \in I $, let $\rho_{i}$ denote the permutation of the vertices of $\cX$ induced by the edges of colour $i$. 
    Denote by $G$ the group $\left\langle \rho_{i} : i \in I  \right\rangle $ and for $i,j \in I $ let $G_{i}$ and $G_{i,j}$ denote the subgroups $\left\langle \rho_{k}: k \in I \sm \left\{ i \right\}  \right\rangle $  and $\left\langle \rho_{k}: k \in I \sm \left\{ i,j \right\}  \right\rangle $, respectively. 
    Assume that for every $i$ the group $G_{i}$ satisfies the intersection  property.
    If for every $i,j \in G$ there exists a vertex $x$ of $\cX$ such that 
    \begin{equation}\label{eq:IP_CPR}
    |x G_i \cap xG_j | |\stab_{G_i}(x) \cap  \stab_{G_j}(x)| \leq |G_{i,j}|,
    \end{equation}
    then $G$ satisfies the intersection property.
    In particular, the condition in \cref{eq:IP_CPR} holds if 
    \begin{equation}\label{eq:IP_gcd}
    |x G_i \cap xG_j | \gcd\left(|\stab_{G_i}(x)|,|\stab_{G_j}(x)|\right) \leq |G_{i,j}|.
    \end{equation}
\end{lemma}
\begin{proof}
    Observe that for every vertex $x$ of $\cX$ and every $i,j \in I$ the following computations follow directly from the Orbit-Stabiliser Theorem:
    \[\begin{aligned} 
    | G_i \cap G_j| 
    &= |x (G_i \cap G_j) | \cdot | \stab_{(G_i \cap G_j)}(x)|  \\
    &=|x (G_i \cap G_j) | \cdot | \stab_{G_i}(x) \cap  \stab_{G_j}(x)| \\
    & \leq |x G_i \cap xG_j | | \stab_{G_i}(x) \cap  \stab_{G_j}(x)|.
    \end{aligned}\]
    If \cref{eq:IP_CPR} holds for some vertex $x$ then 
    \[|G_{i}\cap G_{j}| \leq |G_{ij}|.\] 
    Since $G_{i,j} \subset G_{i} \cap G_{j}$, then $G_{i,j} = G_{i} \cap G_{j}$. The intersection property follows from \cref{prop:IP}.
\end{proof}

\begin{lemma}\label{lem:FT_CPR}
    Let $\cX$ be a proper $n$-edge-coloured graph with colour-set $I = \left\{ 0, \dots, n-1 \right\} $. 
    For $i \in I $, let $\rho_{i}$ denote the permutation of the vertices of $\cX$ induced by the edges of colour $i$. 
    Denote by $G$ the group $\left\langle \rho_{i} : i \in I  \right\rangle $ and for $i,j \in I $ let $G_{i}$ and $G_{i,j}$ denote the subgroups $\left\langle \rho_{k}: k \in I \sm \left\{ i \right\}  \right\rangle $  and $\left\langle \rho_{k}: k \in I \sm \left\{ i,j \right\}  \right\rangle $, respectively. 
    Assume that for every $i \in I$ the group $G_{i}$ is flag-transitive. 
    If for some vertex $x$ of $\cX$ the inequality 
    \begin{equation}\label{eq:FT_CPR}
    \left|\left(  x G_i \cap x G_jG_k \right) \right|  \left|\stab_{G_i}(x)\right| \leq |(G_i \cap G_j)(G_i \cap G_k)|
    \end{equation}
    holds, then $G$ is flag transitive on  the flags of type $\left\{ i,j,k \right\} $ of $\Gamma(G;(G_i)_{i\in I})$.
\end{lemma}
\begin{proof}
    According to \cref{lem:FT_Group}, it is sufficient to show that for every $ \{i,j,k\} \subseteq I$
    \[ G_i \cap G_jG_k = (G_i \cap G_j)(G_i \cap G_k). \]
    One inclusion is obvious,  to proof the other we only need to show that
    \[ |G_i \cap G_jG_k| \leq |(G_i \cap G_j)(G_i \cap G_k)|. \]
    
    Take any vertex $x$ of $\cX$ and consider the set $x(G_i \cap G_jG_k) = \{x \alpha : \alpha \in G_i \cap G_jG_k\}$. We show that
    \[ |G_i \cap G_jG_k| \leq |x(G_i \cap G_jG_k)| | \stab_{G_i}(x)|.\]

    To see this, just observe that 	
    \[ G_i \cap G_jG_k = \bigcup_{y \in x(G_i \cap G_jG_k)} \{\alpha \in G_i \cap G_jG_k : x \alpha = y\},\]
    but for a given $y \in x(G_i \cap G_jG_k)$ 
    \[ \{\alpha \in G_i \cap G_jG_k : x \alpha = y\} \subseteq \{\alpha \in G_i : x \alpha = y\} \]
    and the set $\{\alpha \in G_i : x \alpha = y\}$ has size $|\stab_{G_i}(x)|$.

    Finally, note that $x\left( G_{i} \cap G_{j}G_{k} \right) \subset \left( x G_{i} \cap x (G_{j}G_{k} \right)$.
    Hence, if the condition in \cref{eq:FT_CPR} holds for some vertex $x$, then 
    \[ |G_i \cap G_jG_k| \leq |\left(xG_i \cap x(G_jG_k)\right)| | \stab_{G_i}(x)| \leq |(G_i \cap G_j)(G_i \cap G_k)|. \]
\end{proof}

Observe that when applying \cref{lem:FT_Group} to the subgroups $G_{i}$, $G_{j}$, and $G_{k}$ the role of $i$, $j$ and $k$ is symmetric. 
However, for \cref{lem:FT_CPR} the index $i$ plays a different role than that of $j$ and $k$, meaning that the condition in \cref{eq:FT_CPR} might be true for a given order of $\left\{ i,j,k \right\} $ but not for every one of them.
For this reason when using the \cref{lem:FT_CPR} such ordered triples will be denoted by $(i,j,k)$. 
 \section{Constructions}\label{sec:constructions}

In this section for each non-linear Coxeter diagram $D$ of hyperbolic type and each positive integer $t$ we give an explicit construction of a proper edge-coloured graph $\cX^{t}_{D}$.
The permutation group induced by each of these graphs is shown to be the type-preserving automorphism group of a regular hypertope $\cH_{D}^{t}$ whose Coxeter diagram is $D$ for all but finitely many integers $t$. 
As a consequence we build an infinite family of proper regular hypertopes for each hyperbolic type of rank $4$ and $5$.	

For each type $D$ in ranks $4$ and $5$ we start with a graph $X_{D}=(V_{D},E_{D})$ where the vertex-set is $V_{D}=\left\{ 1, \dots, d \right\} $. 
We take $t$ copies of the graph $X_{D}$ to construct the graph $\cX^{t}_{D}$ whose vertex-set is $V_{D}^{t} = V_{D} \times \bZ_{t}$ and such that for every $\ell \in \bZ_{t}$, the subgraph induced by the vertices $\left\{ (v, \ell): v \in V_{D} \right\} $ is isomorphic to $X_{D}$. 
We denote these subgraphs by $X^{\ell}_{D}$.
In each case we specify how to add edges connecting (some) vertices in $X_{D}^{\ell}$ with some vertices in $X_{D}^{\ell+1}$.
The graph $\cX_{D}^{t}$ can be also construced as derived graphs of voltage assigments on $\bZ_{t}$ (see  \cite{Gross_1974_VoltageGraphs,MalnicNedelaSkoviera_2000_LiftingGraphAutomorphisms}).

We show that for most values of $t$ (usually for $t \geq 1$ or $t \geq 2$),  the permutation group $G$ induced by the graph $\cX^{t}_{D}$ satisfies the intersection property and that it is flag-transitive. 
This implies that there exists a regular hypertope $\cH_{D}^{t}$ whose type-preserving automorphism group is $G$.
As a consequence of these constructions, we find an infinite family of proper regular hypertopes for each hyperbolic type. 

In \cref{sec:3334} the constructions and the proofs are explained in great detail. We will omit some of those details in subsequent sections where the techniques are similar and can easily be adapted to each particular type.

\subsection{Type \texorpdfstring{$( 3,3,3,4) $}{(3,3,3,4)}.} \label{sec:3334}

\begin{figure}
\def\svgwidth{.9\textwidth}
\begingroup \makeatletter \providecommand\color[2][]{\errmessage{(Inkscape) Color is used for the text in Inkscape, but the package 'color.sty' is not loaded}\renewcommand\color[2][]{}}\providecommand\transparent[1]{\errmessage{(Inkscape) Transparency is used (non-zero) for the text in Inkscape, but the package 'transparent.sty' is not loaded}\renewcommand\transparent[1]{}}\providecommand\rotatebox[2]{#2}\newcommand*\fsize{\dimexpr\f@size pt\relax}\newcommand*\lineheight[1]{\fontsize{\fsize}{#1\fsize}\selectfont}\ifx\svgwidth\undefined \setlength{\unitlength}{1224.01330566bp}\ifx\svgscale\undefined \relax \else \setlength{\unitlength}{\unitlength * \real{\svgscale}}\fi \else \setlength{\unitlength}{\svgwidth}\fi \global\let\svgwidth\undefined \global\let\svgscale\undefined \makeatother \begin{picture}(1,0.29526306)\lineheight{1}\setlength\tabcolsep{0pt}\put(0,0){\includegraphics[width=\unitlength,page=1]{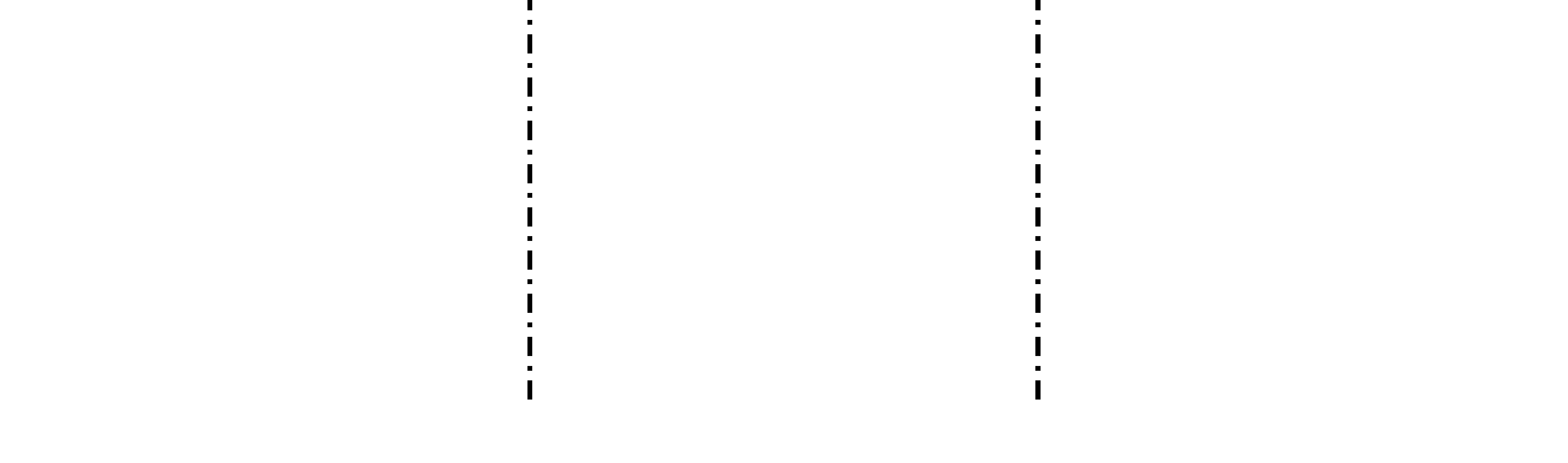}}\put(0.49999999,0.00578033){\color[rgb]{0,0,0}\makebox(0,0)[t]{\lineheight{1.25}\smash{\begin{tabular}[t]{c}$X_{(3,3,3,4)}^{\ell}$\end{tabular}}}}\put(0.82422062,0.00578033){\color[rgb]{0,0,0}\makebox(0,0)[t]{\lineheight{1.25}\smash{\begin{tabular}[t]{c}$X_{(3,3,3,4)}^{\ell+1}$\end{tabular}}}}\put(0.17577936,0.00578033){\color[rgb]{0,0,0}\makebox(0,0)[t]{\lineheight{1.25}\smash{\begin{tabular}[t]{c}$X_{(3,3,3,4)}^{\ell-1}$\end{tabular}}}}\put(0,0){\includegraphics[width=\unitlength,page=2]{3334.pdf}}\end{picture}\endgroup  \caption{The graph $\cX_{(3,3,3,4)}^{t}$.}
\label{fig:3334}
\end{figure}	

Following the notation introduced above, for the diagram
\begin{equation} \label{eq:cox3334}
\begin{tikzcd}[column sep=small]
\overset{\rho_{0}} {\textcolor{Red}{\bullet}} \arrow[d, dash] \arrow[r, dash, "4"]
& \overset{\rho_{3}}{\textcolor{Orange}{\bullet}} \arrow[d, dash] \\
\underset{\rho_{1}}{\textcolor{Green}\bullet} \arrow[r, dash,]
& \underset{\rho_{2}}{\textcolor{Blue}{\bullet}}
\end{tikzcd},
\end{equation}
we start with the graph $X_{(3,3,3,4)}$ on twenty vertices labelled with the numbers $ 1, \dots, 20 $ and solid edges, isomorphic to the graph $X^{\ell}_{(3,3,3,4)}$ between two vertical lines in \cref{fig:3334}.
Taking $t$ copies of the graph in a cyclical arrangement and connecting the vertices of consecutive copies $X^{\ell}_{(3,3,3,4)}$ and $X^{\ell+1}_{(3,3,3,4)}$, with edges indicated by the dotted edges in \cref{fig:3334}, we obtain the proper $4$-edge-coloured graph $\cX_{(3,3,3,4)}^{t}$ on $20t$ vertices. 
The four generators of the group $G$ induced by the graph are given by the permutations

\begin{gather*}
\begin{multlined}
    \rho_{0} = 
    \prod_{\ell \in \bZ_{t}} 
    \Big(  
        \big( ( 5, \ell) (6 ,\ell)  \big) \cdot
        \big( ( 7, \ell) (9 ,\ell)  \big) \cdot
        \big( ( 8, \ell) (10 ,\ell)  \big) \cdot 
        \big( ( 11, \ell) (13 ,\ell)  \big) \cdot \\ \cdot
        \big( ( 12, \ell) (14 ,\ell)  \big) \cdot
        \big( ( 15, \ell) (16 ,\ell)  \big) \cdot
        \big( ( 19, \ell) (2 ,\ell+1)  \big)
    \Big), 
\end{multlined}\\
\begin{multlined}
    \rho_{1} = 
    \prod_{\ell \in \bZ_{t}}
    \Big(  
        \big( (2 , \ell) (6 ,\ell)  \big) \cdot
        \big( (3 , \ell) (8 ,\ell)  \big) \cdot
        \big( (4 , \ell) (7 ,\ell)  \big) \cdot 
        \big( (13 , \ell) (18 ,\ell)  \big) \cdot \\ \cdot
        \big( (14 , \ell) (17 ,\ell)  \big) \cdot
        \big( (15 , \ell) (19 ,\ell)  \big)\cdot
        \big( (16 , \ell) (5 ,\ell+1)  \big)
    \Big),
\end{multlined}\\
\begin{multlined}
    \rho_{2} = 
    \prod_{\ell \in \bZ_{t}} 
    \Big(  
        \big( (1 , \ell) (3 ,\ell)  \big)\cdot
        \big( (5 , \ell) (7 ,\ell)  \big)\cdot
        \big( (6 , \ell) (9 ,\ell)  \big)\cdot 
        \big( (12 , \ell) (15 ,\ell)  \big)\cdot \\ \cdot
        \big( (14 , \ell) (16 ,\ell)  \big)\cdot
        \big( (18 , \ell) (20 ,\ell)  \big)\cdot
        \big( (17 , \ell) (4 ,\ell+1)  \big) 
    \Big),
\end{multlined}\\
\begin{multlined}
    \rho_{3} = 
    \prod_{\ell \in \bZ_{t}} 
    \Big(  
        \big( (3 , \ell) (4 ,\ell)  \big) \cdot
        \big( (7 , \ell) (8 ,\ell)  \big) \cdot
        \big( (9 , \ell) (11 ,\ell)  \big)\cdot 
        \big( (10 , \ell) (12 ,\ell)  \big) \cdot \\ \cdot
        \big( (13 , \ell) (14 ,\ell)  \big) \cdot
        \big( (17 , \ell) (18 ,\ell)  \big)\cdot
        \big( (20 , \ell) (1 ,\ell+1)  \big) 
    \Big).
\end{multlined}
\end{gather*}

To see that the permutation group $G$ described above satisfies the relations implicit in the Coxeter diagram in \cref{eq:cox3334} we need to verify that each $\left\{ i,j \right\} $-component in the graph $\cX_{(3,3,3,4)}$ is of the correct size. 
For example, every $\left\{ 0,3 \right\}$-component is either an alternating cycle with $8$ vertices, a path of colour $0$ with two vertices, or a path of colour $3$ with two vertices, which implies that the order of $\rho_{0}\rho_{3}$ is $\lcm(4,2,2) = 4$ (see \cref{lem:polygonalAction}).

Note that the orbit of $(7,0)$ under the element $\chi_{1}=\rho_{0} \rho_{3} \rho_{2}$ has $6$ elements, namely $(7,0),(11,0),(16,0),(12,0),(13,0)$ and $(6,0)$.
It follows that the period of $\chi_{1}$ is at least $6$ and therefore the group $G_{1}=\left\langle \rho_{0}, \rho_{3}, \rho_{2} \right\rangle $ is isomorphic to the Coxeter group $[4,3]$ of order $48$.
This implies that the residue of type $1$ is isomorphic to a cube and not a hemi-cube.
Similarly, the orbit of $(7,0)$ under $\chi_{2}=\rho_{3}\rho_{0}\rho_{1}$ is also of length $6$, which implies that the group $G_{2} = \left\langle \rho_{3}, \rho_{0}, \rho_{1} \right\rangle $ is also isomorphic to $[4,3]$. 
Finally, observe that the groups $G_{0} = \left\langle \rho_{1}, \rho_{2}, \rho_{3} \right\rangle $ and $G_{3} = \left\langle \rho_{0}, \rho_{1}, \rho_{2} \right\rangle$ are both isomorphic to the Coxeter group $[3,3]$.
In particular, the subgroups $G_{i}$ satisfy the intersection property for every $i \in \left\{ 0,1,2,3 \right\} $.

We use \cref{lem:IP_CPR} to prove the intersection property for the group $G$. 
In order to do so, for every pair $\left\{ i,j \right\} \subset \left\{ 0,1,2,3 \right\}  $ we need to find a vertex $x_{i,j}$ of the graph $\cX_{(3,3,3,4)}$ that satisfies \cref{eq:IP_CPR} or rather, the slightly stronger condition in \cref{eq:IP_gcd}, that is
\[
|x_{i,j} G_i \cap x_{i,j} G_j | \cdot \gcd\left(  	|\stab_{G_i}(x_{i,j})|, |\stab_{G_j}(x_{i,j})| \right) \leq |G_{i,j}|.
\]

We list those vertices in \cref{tab:IP_3334}, where $o_{i,j}$, $s_{i}$ and $s_{j}$ denote $|x_{ij}G_{i} \cap x_{ij}G_{j}|$, $|\stab_{G_{i}}\left( x_{i,j} \right)|$ and $|\stab_{G_{j}}\left( x_{i,j} \right)|$, respectively. 
Observe that those numbers can be easily computed from the graph.
Verification of the inequality above follows directly from the diagram in \cref{eq:cox3334} and the values in \cref{tab:IP_3334}.

\begin{table}
\begin{tabularx}{\textwidth}{| \cc{.3} | \cc{.3} | \cc{.7} | \cc{2.8} | \cc{2.8} | \cc{.5} | \cc{.3} | \cc{.3} | }
    \hline
    $ i$ & $ j$ & $ x_{i,j}$ & $x_{i,j} G_{i}$ & $ x_{i,j} G_{j}$ & $o_{i,j}$ & $ s_{i}$ & $s_{j} $ \\ \hline
    $0$ & $1$ & $(2,0)$ 
    & $
    \begin{multlined}	
        \left\{ (9, 0), (2, 0), \right. \\\left.(11, 0), (6, 0) \right\}  
    \end{multlined}$
    & $\left\{ (2, 0), (19, -1) \right\} $ 
    & $1$ & $6$ & $24$  \\ \hline
$0$ & $2$ & $(1,0)$ 
    &$	\begin{aligned}
    \{&(1, 0), (3, 0), (4, 0), \\ &(5, 0), (7, 0), (8, 0), \\&(13, -1), (14, -1), \\   &(16, -1), (17, -1), \\&(18, -1), (20, -1) \}
    \end{aligned}$
    & $\left\{ (20, -1), (1, 0) \right\} $ 
    & $2$ & $2$ & $24$   \\ \hline
    $0$ & $3$ & $(1,0)$ 
    & $	\begin{aligned}
    \{&(1, 0), (3, 0), (4, 0), \\ &(5, 0), (7, 0), (8, 0), \\&(13, -1), (14, -1), \\   &(16, -1), (17, -1), \\&(18, -1), (20, -1) \}
    \end{aligned}$
    & $\begin{multlined}
    \left\{  (8, 0), (1, 0), \right. \\ \left. (10, 0), (3, 0)\right\} 
        \end{multlined}$
    & $3$ & $2$ & $6$  \\ \hline
    $1$ & $2$ & $(3,0)$ 
    & $\begin{aligned}
    \{ &(1, 0), (3, 0), \\&(4, 0), (17, -1), \\&(18, -1), (20, -1)  \} 
    \end{aligned}$ 
    & $\begin{aligned}
    \{ &(3, 0), (4, 0), (7, 0), \\&(8, 0), (9, 0), (10, 0), \\&(11, 0), (12, 0), (13, 0), \\&(14, 0), (17, 0), (18, 0) \} 
    \end{aligned}$ 
    & $2$ & $8$ & $4$   \\ \hline
    $1$ & $3$ & $(1,0)$ 
    & $ \begin{aligned}
    \{ &(1, 0), (3, 0), \\ &(4, 0), (17, -1), \\ &(18, -1), (20, -1) \} 
    \end{aligned}$
    & $ \begin{multlined}
    \{ (8, 0), (1, 0), \\ (10, 0), (3, 0) \} 
    \end{multlined}$ 
    & $2$ & $8$ & $6$  \\ \hline
    $2$ & $3$ & $(1,0)$ 
    & $ \begin{aligned}
    \{ &(20, -1), (1, 0) \} 
    \end{aligned}$
    & $ \begin{multlined}
    \{ (8, 0), (1, 0), \\ (10, 0), (3, 0) \} 
    \end{multlined}$ 
    & $1$ & $24$ & $6$  \\ \hline

\end{tabularx}
\caption{Intersection property for $\cX_{(3,3,3,4)}^{t}$}
\label{tab:IP_3334}
\end{table}

To show that the group $G$ is flag-transitive we use \cref{lem:FT_CPR}. 
The approach is very similar to that used for the intersection property. 
For every set $\left\{ i,j,k \right\} \subset \left\{ 0,1,2,3 \right\}  $ we need to find a vertex $x$ that satisfies the condition in \cref{eq:FT_CPR}. 
We list those vertices in \cref{tab:FT_3334} where $o_{i,j,k}$ denotes $|xG_{i} \cap xG_{j}G_{k}|$ and $s_{i}$ denotes $|\stab_{G_{i}}(x)|$. 
Observe that since $G$ satisfies the intersection property 
\[\left| \left( G_{i}\cap G_{j} \right)\left( G_{i} \cap G_{k} \right) \right|
=|G_{i,j} G_{i,k}| 
= \frac{|G_{i,j}||G_{i,k}|}{\left| \left( G_{i,j} \cap G_{i,k} \right) \right| } = 2p_{l,k}p_{l,j}
\]
where for $l \in \left\{ 0,1,2,3 \right\}\sm\left\{ i,j,k \right\}  $, $p_{l,k}$ and $p_{l,j}$ denote the periods of $\rho_{l} \rho_{k}$ and $\rho_{l} \rho_{j}$, respectively.

\begin{table}
\begin{tabularx}{\textwidth}
{| \cc{.3} | \cc{.3} |  \cc{.3} | \cc{.7} | \cc{2.8} | \cc{2.8} | \cc{.5} | \cc{.3} | }
    \hline
    $i$ & $j$ & $k$ & $x$& $xG_{i}$ & $xG_{j}G_{k}$ & $o_{i,j,k}$ & $s_{i}$ \\ \hline
    $0$ & $1$ & $2$ & $(2,0)$  
    & $ 
\begin{multlined}
    \{ (9, 0), (2, 0), \\ (11, 0), (6, 0) \} 
    \end{multlined}
$
    & $ 
    \begin{aligned}
\{ &(2, 0), (5, 0), \\ &(6, 0), (15, -1), \\ &(16, -1), (19, -1) \} 
\end{aligned}
    $
    & $2$ & $6$ \\ \hline
    $0$ & $1$ & $3$ & $(2,0)$  
    & $ 
\begin{multlined}
    \{ (9, 0), (2, 0), \\ (11, 0), (6, 0) \} 
    \end{multlined}
$
    & $ 
    \begin{aligned}
\{ &(2, 0), (4, 0), (5, 0), \\ &(6, 0), (7, 0), (12, -1), \\ &(9, 0), (14, -1), \\& (15, -1),  (16, -1), \\& (17, -1), (19, -1) \} 
\end{aligned}
    $
    & $3$ & $6$ \\ \hline
    $0$ & $2$ & $3$ & $(1,0)$  
    & $ 
    \begin{aligned}
\{ &(1, 0), (3, 0), (4, 0), \\ &(5, 0), (7, 0), (8, 0), \\&(13, -1), (14, -1), \\&(16, -1), (17, -1),\\& (18, -1), (20, -1) \} 
\end{aligned}
    $
    & $ 
    \begin{aligned}
\{&(1, 0), (3, 0), (11, -1), \\&(8, 0), (13, -1), (10, 0), \\&(18, -1), (20, -1)  \} 
\end{aligned}
    $
    & $6$ & $2$ \\ \hline
    $1$ & $2$ & $3$ & $(5,0)$  
    & $ 
    \begin{aligned}
\{  &(5, 0), (6, 0), (7, 0),\\ &(8, 0), (9, 0), (10, 0), \\&(11, 0), (12, 0), (13, 0),\\ &(14, 0), (15, 0), (16, 0)\} 
\end{aligned}
    $
    & $ 
    \begin{aligned}
\{ &(2, 0), (4, 0), (5, 0), \\&(6, 0), (7, 0), (12, -1), \\&(9, 0), (14, -1), \\&(15, -1), (16, -1), \\&(17, -1), (19, -1) \} 
\end{aligned}
    $
    & $4$ & $4$ \\\hline
\end{tabularx}	\caption{Flag transitivity for $\cX_{(3,3,3,4)}^{t}$}
\label{tab:FT_3334}
\end{table}
    
In both arguments above, the one used to prove the intersection property and the one for flag-transitivity, we are strongly using the fact that $t \geq 2$.
Our construction is well-defined for $t=1$ and the intersection property and flag transitivity can be easily checked using, for example, SageMath	 \cite{Developers_2020_SagemathSageMathematics}.

\subsection{Type \texorpdfstring{$( 3,3,3,5) $}{(3,3,3,5)}.} \label{sec:3335}

\begin{figure}
\def\svgwidth{.9\textwidth}
\begingroup \makeatletter \providecommand\color[2][]{\errmessage{(Inkscape) Color is used for the text in Inkscape, but the package 'color.sty' is not loaded}\renewcommand\color[2][]{}}\providecommand\transparent[1]{\errmessage{(Inkscape) Transparency is used (non-zero) for the text in Inkscape, but the package 'transparent.sty' is not loaded}\renewcommand\transparent[1]{}}\providecommand\rotatebox[2]{#2}\newcommand*\fsize{\dimexpr\f@size pt\relax}\newcommand*\lineheight[1]{\fontsize{\fsize}{#1\fsize}\selectfont}\ifx\svgwidth\undefined \setlength{\unitlength}{1266.53308105bp}\ifx\svgscale\undefined \relax \else \setlength{\unitlength}{\unitlength * \real{\svgscale}}\fi \else \setlength{\unitlength}{\svgwidth}\fi \global\let\svgwidth\undefined \global\let\svgscale\undefined \makeatother \begin{picture}(1,0.25177884)\lineheight{1}\setlength\tabcolsep{0pt}\put(0,0){\includegraphics[width=\unitlength,page=1]{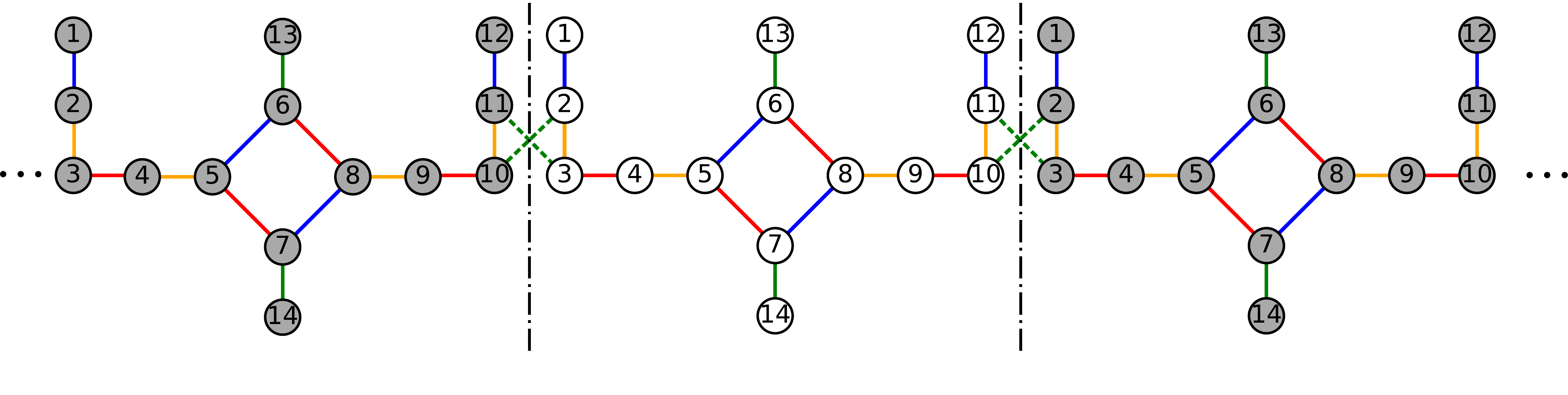}}\put(0.49440467,0.00558627){\color[rgb]{0,0,0}\makebox(0,0)[t]{\lineheight{1.25}\smash{\begin{tabular}[t]{c}$X_{(3,3,3,5)}^{\ell}$\end{tabular}}}}\put(0.1810687,0.00558627){\color[rgb]{0,0,0}\makebox(0,0)[t]{\lineheight{1.25}\smash{\begin{tabular}[t]{c}$X_{(3,3,3,5)}^{\ell-1}$\end{tabular}}}}\put(0.80774067,0.00558627){\color[rgb]{0,0,0}\makebox(0,0)[t]{\lineheight{1.25}\smash{\begin{tabular}[t]{c}$X_{(3,3,3,5)}^{\ell+1}$\end{tabular}}}}\end{picture}\endgroup  \caption{The graph $\cX_{(3,3,3,5)}^{t}$.}
\label{fig:3335}
\end{figure}	

The vertex set of the graph $\cX^{t}_{(3,3,3,5)}$ for the diagram in \cref{eq:cox3335} is $\left\{ 1,\dots,14 \right\} \times \bZ_{t} $. 
The edges of the graph $X_{(3,3,3,5)}$ can again be seen as the solid edges in between the vertical lines in \cref{fig:3335}.
The dotted edges connecting the vertices $(10,\ell)$ and $(11,\ell)$ to $(2,\ell+1)$ and $(3,\ell+1)$, respectively, for every $\ell \in \bZ_{t}$ complete the edge-set of $\cX^{t}_{(3,3,3,5)}$.

\begin{equation}\label{eq:cox3335}
\begin{tikzcd}[column sep=small]
\overset{\rho_{0}} {\textcolor{Red}{\bullet}} \arrow[d, dash] \arrow[r, dash, "5"]
& \overset{\rho_{3}}{\textcolor{Orange}{\bullet}} \arrow[d, dash] \\
\underset{\rho_{1}}{\textcolor{Green}\bullet} \arrow[r, dash,]
& \underset{\rho_{2}}{\textcolor{Blue}{\bullet}}
\end{tikzcd}
\end{equation}

As in the previous section, we can use \cref{lem:polygonalAction} to verify that the relations for the permutation group $G$ induced by $\cX^{t}_{(3,3,3,5)}$ are actually those implicit in \cref{eq:cox3335}.
It is easily checked that the order of the Coxeter elements in the groups $G_{1}=\left\langle \rho_{0}, \rho_{3}, \rho_{2} \right\rangle $ and $G_{2}= \left\langle \rho_{3}, \rho_{0}, \rho_{1} \right\rangle $ is $10$. 
It follows that the residues of type $1$ and $2$ are spherical polytopes. The residues of type $0$ and $3$ are obviously spherical.

Analogous to what we did in \cref{sec:3334}, the intersection property for the group $G$ can be proved using \cref{lem:IP_CPR}.
In this case, we can use the vertex $(1,0)$ for every pair $\left\{ i,j \right\} \subset \left\{ 0,1,2,3 \right\} $.
Similarly, we can use \cref{lem:FT_CPR} to prove that the group $G$ is flag-transitive. 
It can be easily verified that \cref{eq:FT_CPR} holds for $x=(13,0)$ when $(i,j,k)$ is $(0,1,2)$ or $(0,1,3)$, and for $x=(1,0)$ when $(i,j,k)$ is $(0,2,3)$ or $(1,2,3)$. 
For both arguments to hold we need that $t\geq 2$. 
Intersection property and flag transitivity for $t=1$ can be easily verified using, for example, SageMath \cite{Developers_2020_SagemathSageMathematics}.

\subsection{Type \texorpdfstring{$( 3,4,3,4) $}{(3,4,3,4)}.} \label{sec:3434}

\begin{figure}
\def\svgwidth{.9\textwidth}
\begingroup \makeatletter \providecommand\color[2][]{\errmessage{(Inkscape) Color is used for the text in Inkscape, but the package 'color.sty' is not loaded}\renewcommand\color[2][]{}}\providecommand\transparent[1]{\errmessage{(Inkscape) Transparency is used (non-zero) for the text in Inkscape, but the package 'transparent.sty' is not loaded}\renewcommand\transparent[1]{}}\providecommand\rotatebox[2]{#2}\newcommand*\fsize{\dimexpr\f@size pt\relax}\newcommand*\lineheight[1]{\fontsize{\fsize}{#1\fsize}\selectfont}\ifx\svgwidth\undefined \setlength{\unitlength}{757.37719727bp}\ifx\svgscale\undefined \relax \else \setlength{\unitlength}{\unitlength * \real{\svgscale}}\fi \else \setlength{\unitlength}{\svgwidth}\fi \global\let\svgwidth\undefined \global\let\svgscale\undefined \makeatother \begin{picture}(1,0.43508292)\lineheight{1}\setlength\tabcolsep{0pt}\put(0,0){\includegraphics[width=\unitlength,page=1]{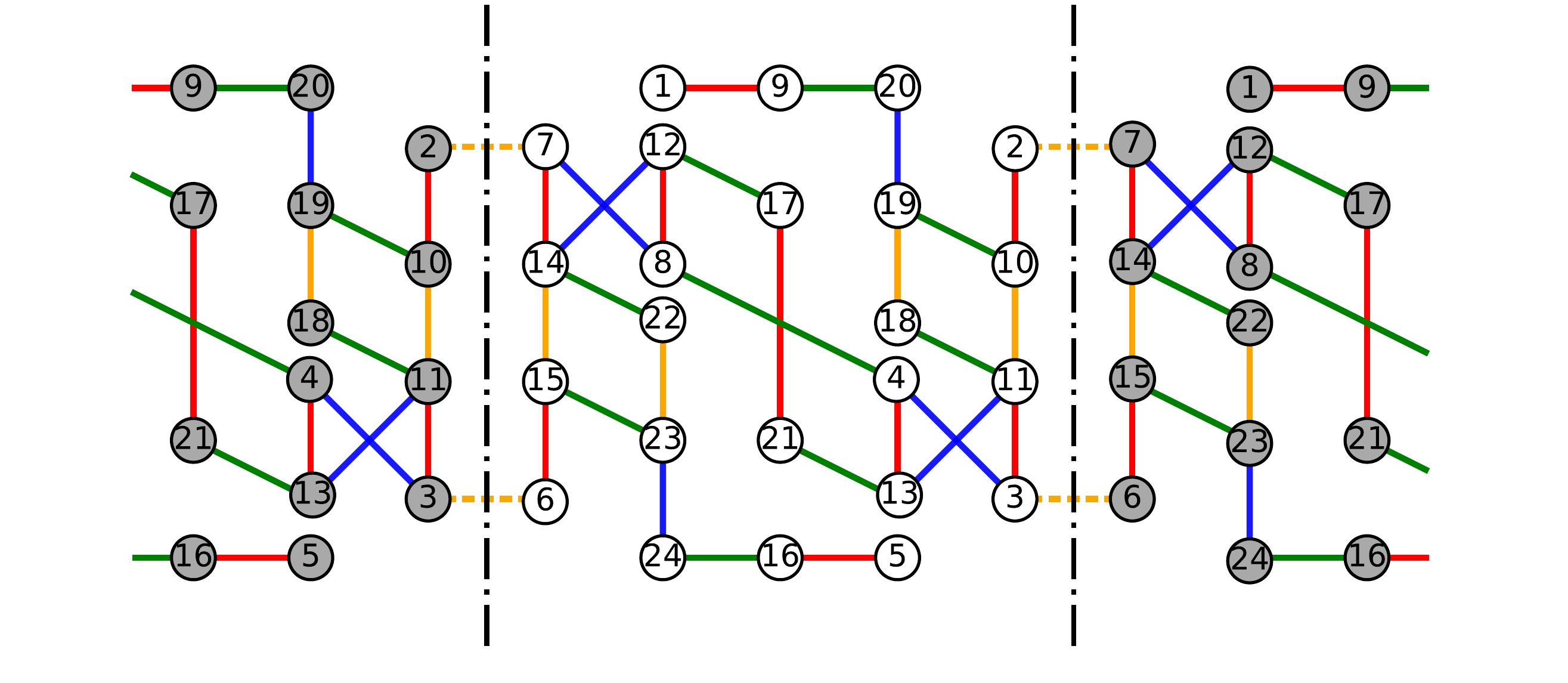}}\put(0.49763799,0.00467086){\color[rgb]{0,0,0}\makebox(0,0)[t]{\lineheight{1.25}\smash{\begin{tabular}[t]{c}$X_{(3,4,3,4)}^{\ell}$\end{tabular}}}}\put(0.12336663,0.00467086){\color[rgb]{0,0,0}\makebox(0,0)[t]{\lineheight{1.25}\smash{\begin{tabular}[t]{c}$X_{(3,4,3,4)}^{\ell-1}$\end{tabular}}}}\put(0.87190936,0.00467086){\color[rgb]{0,0,0}\makebox(0,0)[t]{\lineheight{1.25}\smash{\begin{tabular}[t]{c}$X_{(3,4,3,4)}^{\ell+1}$\end{tabular}}}}\put(0,0){\includegraphics[width=\unitlength,page=2]{3434.pdf}}\end{picture}\endgroup  \caption{The graph $\cX_{(3,4,3,4)}^{t}$.  }
\label{fig:3434}
\end{figure}	

The graph $\cX_{(3,4,3,4)}^{t}$ in \cref{fig:3434} is associated with the Coxeter diagram in \cref{eq:cox3434}. 
\begin{equation}\label{eq:cox3434}
\begin{tikzcd}[column sep=small]
\overset{\rho_{0}} {\textcolor{Red}{\bullet}} \arrow[d, dash] \arrow[r, dash, "4"]
& \overset{\rho_{3}}{\textcolor{Orange}{\bullet}} \arrow[d, dash] \\
\underset{\rho_{1}}{\textcolor{Green}\bullet} \arrow[r, dash, "4"']
& \underset{\rho_{2}}{\textcolor{Blue}{\bullet}}
\end{tikzcd}
\end{equation}

As above, the relations for the permutation group $G$ can be verified using \cref{lem:polygonalAction}.
To see that the maximal residues of the induced hypertope are spherical and not projective, consider the orbit under the action of one of the Coxeter elements on $(1,0)$ for $G_{3}$, on $(11,0)$ for $G_{2}$, on $(14,0)$ for $G_{1}$, and on $(3,0)$ for $G_{0}$.
All those four orbits are of length $6$, which implies that the groups $G_{0}$, $G_{1}$, $G_{2}$ and $G_{3}$ are all isomorphic to the Coxeter group $[4,3]$. 

To prove the intersection property, if $\left\{ i,j \right\} \neq \left\{ 2,3 \right\}  $, we can proceed as before and use \cref{eq:IP_gcd} on the vertices $x_{0,1}=(9,0)$, $x_{0,2}=(9,0)$, $x_{0,3}=(1,0)$, $x_{1,2}=(4,0)$ and $x_{1,3}=(18,0)$ to show that $G_{i} \cap G_{j} = G_{i,j}$.
However, for every vertex $x$ of $\cX_{(3,4,3,4)^{t}}$, 
\[\begin{aligned}
  \gcd\left( \left| \stab_{G_{2}}(x) \right|, \left| \stab_{G_{3}} \right|   \right) &\geq 4 && \text{and} \\ \left| xG_{2} \cap x G_{3} \right| &\geq 2.
\end{aligned}\]
This implies that we cannot use \cref{eq:IP_gcd} for $\{i,j\} = \{2,3\}$. However we still can use \cref{lem:IP_CPR}. In particular we will show that \cref{eq:IP_CPR} holds for $x=(3,0)$.

First, let us show that $\stab_{G_{2}}(x)\cap \stab_{G_{3}}(x) = \left\langle \rho_{1} \right\rangle $. 
One inclusion is obvious, since there is no edge of colour $1$ incident to $(3,0)$ in $\cX_{(3,4,3,4)}$.
Observe that $\stab_{G_{2}}(x)= \left\langle \rho_{1}, \alpha \right\rangle $ where $\alpha=\rho_{0} \rho_{3} \rho_{1} \rho_{0} \rho_{1}\rho_{3} \rho_{0}$. 
To see this, just notice that $\left\langle \rho_{1}, \alpha \right\rangle \leq \stab_{G_{2}}(x)$ and that the orbit of $(11,0)$ under $\left\langle  \rho_{1}, \alpha \right\rangle $ is $\left\{ (11,0), (18,0), (14,1), (21,1)  \right\} $. 
The latter implies that $\left| \left\langle \rho_{1}, \alpha \right\rangle \right| \geq 4 = \left| \stab_{G_{2}}(x) \right| $.
Note that every element in $\stab_{G_{2}}(x)\cap \stab_{G_{3}}(x)$ permutes the elements in the set  $xG_{2} \cap xG_{3} = \left\{ (3,0), (11,0),(18,0) \right\} $ but $(11,0) \alpha = (14,1)$. 
Therefore $\stab_{G_{2}}(x)\cap \stab_{G_{3}}(x)$ is a proper subgroup of $\left\langle \rho_{1}, \alpha \right\rangle $, hence of order at most $2$.

As a consequence of the discussion above, the vertex $x=(3,0)$ satisfies \cref{eq:IP_CPR}, that is, 
\[\left| xG_{2} \cap xG_{3} \right|\cdot \left| \stab_{G_{2}}(x)\cap \stab_{G_{3}}(x) \right| = 3 \cdot 2 = \left| G_{2,3} \right|.\]
By \cref{lem:IP_CPR}, the group $G$ satisfies the intersection property.

Flag transitivity follows in a similar way as in \cref{sec:3334} from \cref{lem:FT_CPR} by considering the vertices $x$ as follows:
\begin{itemize}
    \item $(9,0)$ when  $(i,j,k) $ is $(0,1,2)$ or $(0,1,3)$,
    \item $(1,0)$ for $(i,j,k)=(3,0,2)$, and
    \item $(4,0)$ for $(i,j,k)=(1,2,3)$.
\end{itemize}

For our arguments proving the intersection property and flag transitivity for the permutation group to hold, we are assuming $t\geq 2$. 
The intersection property and flag transitivity can be easily checked for $t=1$ using computational tools.

\subsection{Type \texorpdfstring{$( 3,4,3,5) $}{(3,4,3,5)}.} \label{sec:3435}

\begin{figure}
\def\svgwidth{\textwidth}
\begingroup \makeatletter \providecommand\color[2][]{\errmessage{(Inkscape) Color is used for the text in Inkscape, but the package 'color.sty' is not loaded}\renewcommand\color[2][]{}}\providecommand\transparent[1]{\errmessage{(Inkscape) Transparency is used (non-zero) for the text in Inkscape, but the package 'transparent.sty' is not loaded}\renewcommand\transparent[1]{}}\providecommand\rotatebox[2]{#2}\newcommand*\fsize{\dimexpr\f@size pt\relax}\newcommand*\lineheight[1]{\fontsize{\fsize}{#1\fsize}\selectfont}\ifx\svgwidth\undefined \setlength{\unitlength}{1654.96875bp}\ifx\svgscale\undefined \relax \else \setlength{\unitlength}{\unitlength * \real{\svgscale}}\fi \else \setlength{\unitlength}{\svgwidth}\fi \global\let\svgwidth\undefined \global\let\svgscale\undefined \makeatother \begin{picture}(1,0.55000795)\lineheight{1}\setlength\tabcolsep{0pt}\put(0,0){\includegraphics[width=\unitlength,page=1]{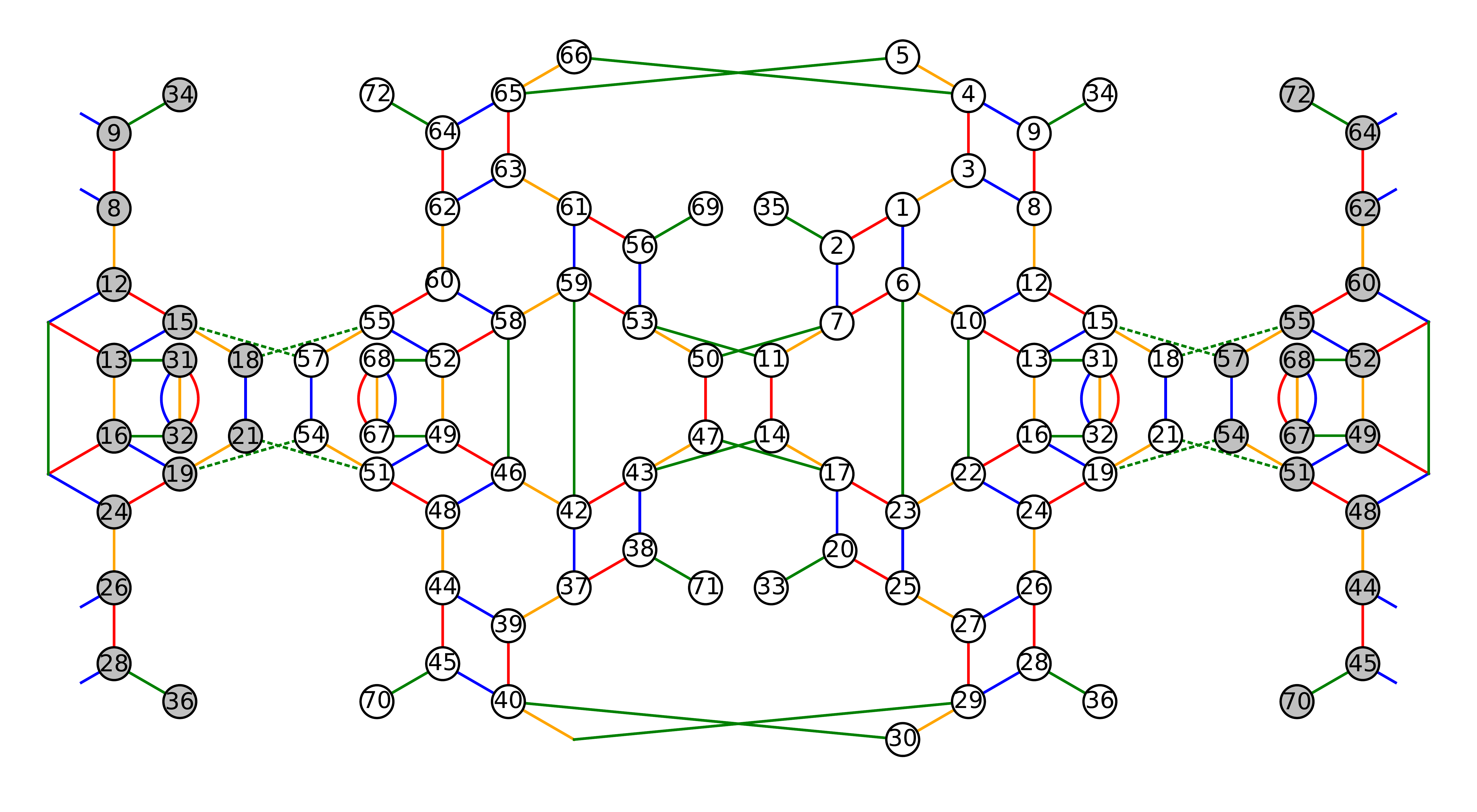}}\put(0.50000002,0.00213758){\color[rgb]{0,0,0}\makebox(0,0)[t]{\lineheight{1.25}\smash{\begin{tabular}[t]{c}$X_{(3,4,3,5)}^{\ell}$\end{tabular}}}}\put(0.07724916,0.00213758){\color[rgb]{0,0,0}\makebox(0,0)[t]{\lineheight{1.25}\smash{\begin{tabular}[t]{c}$X_{(3,4,3,5)}^{\ell-1}$\end{tabular}}}}\put(0.92275087,0.00213758){\color[rgb]{0,0,0}\makebox(0,0)[t]{\lineheight{1.25}\smash{\begin{tabular}[t]{c}$X_{(3,4,3,5)}^{\ell+1}$\end{tabular}}}}\put(0,0){\includegraphics[width=\unitlength,page=2]{3435.pdf}}\end{picture}\endgroup  \caption{The graph $\cX_{(3,4,3,5)}^{t}$.  }
\label{fig:3435}
\end{figure}

In \cref{fig:3435} we introduce the proper $4$-edge-coloured graph  $\cX_{(3,4,3,5)}$ associated with the hypertopes whose Coxeter diagram is the one in \cref{eq:cox3435}.
\begin{equation}\label{eq:cox3435}
\begin{tikzcd}[column sep=small]
    \overset{\rho_{0}} {\textcolor{Red}{\bullet}} \arrow[d, dash] \arrow[r, dash, "5"]
        & \overset{\rho_{3}}{\textcolor{Orange}{\bullet}} \arrow[d, dash] \\
        \underset{\rho_{1}}{\textcolor{Green}\bullet} \arrow[r, dash, "4"']
        & \underset{\rho_{2}}{\textcolor{Blue}{\bullet}}
    \end{tikzcd}
\end{equation}

As before, the relations implicit in the diagram in \cref{eq:cox3435} for the induced permutation group $G$ follow from \cref{lem:polygonalAction}. 

The intersection property for the group $G$ follows from \cref{eq:IP_gcd} with the vertices $x_{i,j}$ as follows:
\begin{itemize}
    \item $(1,0)$ if $\{i,j\}$ is $\{0,2\}$, $\{1,2\}$ or $\{2,3\}$, 
    \item $(2,0)$ if $\{i,j\}$ is $\{0,1\}$ or $\{0,3\}$, and
    \item $(5,0)$ if $\{i,j\} = \{0,2\} $.
\end{itemize}

We can use \cref{eq:FT_CPR} with the vertex $(31,0)$ when $(i,j,k)$ is $(0,1,2)$, $(0,1,3)$ and $(2,1,3)$ to prove that \[G_{i} \cap G_{j}G_{k} = G_{i,j}G_{i,k}. \]
However, if $\{i,j,k\} = \{0,2,3\}$ there is no vertex $x$ of the graph $\cX_{(3,4,3,5)}$ that satisfies \cref{eq:FT_CPR}, but the following argument proves that the subgroups $G_{0}$, $G_{2}$ and $G_{3}$ satisfy \cref{lem:FT_Group}.

Observe that \[G_{i} \cap G_{j}G_{k}  = G_{i} \cap \left( \bigcup_{\alpha \in R} \alpha G_{k} \right) \] where $R$ is a set of coset representatives of $G_{j,k}$ in $G_{j}$. If $G_{j}$ is a Coxeter group then the set $R$ is easy to compute (see \cite[Section 1.10]{Humphreys_1990_ReflectionGroupsCoxeter}).
If $i=2$, $j=3$ and $k=0$ then $G_{j} = \left\langle \rho_{2}, \rho_{1}, \rho_{0} \right\rangle$ is isomorphic to the Coxeter group $[4,3]$.

The subgroup $G_{j,k}$ has index $6$ in $G_{j}$ and the set $R$ of coset representatives can be chosen as $\{\id, \rho_{0}, \rho_{1}\rho_{0}, \rho_{2}\rho_{1}\rho_{0}, \rho_{1}\rho_{2}\rho_{1}\rho_{0}, \rho_{0}\rho_{1}\rho_{2}\rho_{1}\rho_{0}\}$.
Since $G$ satisfies the intersection property, we know that $|G_{2} \cap G_{0}| = |G_{2,0}|=4$.
Now, since $\rho_{0} \in G_{2}$, \[|G_{2} \cap \rho_{0} G_{0}| = |\rho_{0}(G_{2} \cap G_{0})| =|G_{2,0}|=4.\]
Similarly $|G_{2} \cap \rho_{1}\rho_{0} G_{0}| = 4$. 
Since $|G_{2,3}G_{2,0}| = \frac{|G_{2,3}|\cdot|G_{2,0}|}{|(G_{2,3}\cap G_{2,0})|} = \frac{6\cdot 4}{2} = 12$, we need to show that $G_{2} \cap \alpha G_{0} = \emptyset$ for $\alpha \in \{ \rho_{2}\rho_{1}\rho_{0}, \rho_{1}\rho_{2}\rho_{1}\rho_{0}, \rho_{0}\rho_{1}\rho_{2}\rho_{1}\rho_{0}\}$.

Assume that $\alpha = \rho_{2}\rho_{1}\rho_{0}$ and let $x$ be the vertex $(1,0)$ of $\cX_{(3,4,3,5)}$ and let $y = x \alpha = (17,0)$.
Observe that
\[\begin{aligned}
&\begin{multlined}
xG_{2} = \{ (1,0), (2,0),  (3,0),  (4,0), (35, 0), (66, 0), \\(5, 0), (65, 0), (69, 0), (56, 0), (61, 0), (63, 0)\}
\end{multlined} && \text{and} \\ 
&\begin{multlined}
yG_{0} = \{ (33, 0), (38, 0), (71, 0), (43, 0), \\ (14, 0), (47, 0), (17, 0), (20, 0)\}
\end{multlined}
\end{aligned}\]

An element in  $G_{2} \cap \alpha G_{0}$ must map $x$ to a vertex in \[xG_{2} \cap (x\alpha)G_{0} = xG_{2} \cap yG_{0}.\] 
Since $xG_{2} \cap yG_{0} = \emptyset$, we have $G_{2} \cap \rho_{2}\rho_{1}\rho_{0} G_{0} = \emptyset$.
Now, because $\rho_{1}$ and $\rho_{0}\rho_{1}$ are elements in $G_{2}$, it follows that 
\[|G_{2}\cap \rho_{0}\rho_{1}\rho_{2}\rho_{1}\rho_{0}G_{0}|=|G_{2}\cap \rho_{1}\rho_{2}\rho_{1}\rho_{0}G_{0}|=|G_{2}\cap \rho_{2}\rho_{1}\rho_{0}G_{0}|=0.\]
Therefore 
\[G_{2} \cap G_{3}G_{0} = G_{2,3}G_{2,0}. \]

\cref{thm:FT_rank3} implies that the group $G$ is flag-transitive.
Therefore the graph in \cref{fig:3435} is in fact a CPR-graph of a regular hypertope. 
As in previous sections, the arguments presented here hold if $t \geq 2$.
If $t=1$ the induced group also satisfies the intersection property and is flag-transitive.
These conditions can be easily checked with SageMath.

\subsection{Type \texorpdfstring{$( 3,5,3,5) $}{(3,5,3,4)}.} \label{sec:3535}

\begin{figure}
\def\svgwidth{\textwidth}
\begingroup \makeatletter \providecommand\color[2][]{\errmessage{(Inkscape) Color is used for the text in Inkscape, but the package 'color.sty' is not loaded}\renewcommand\color[2][]{}}\providecommand\transparent[1]{\errmessage{(Inkscape) Transparency is used (non-zero) for the text in Inkscape, but the package 'transparent.sty' is not loaded}\renewcommand\transparent[1]{}}\providecommand\rotatebox[2]{#2}\newcommand*\fsize{\dimexpr\f@size pt\relax}\newcommand*\lineheight[1]{\fontsize{\fsize}{#1\fsize}\selectfont}\ifx\svgwidth\undefined \setlength{\unitlength}{902.09814453bp}\ifx\svgscale\undefined \relax \else \setlength{\unitlength}{\unitlength * \real{\svgscale}}\fi \else \setlength{\unitlength}{\svgwidth}\fi \global\let\svgwidth\undefined \global\let\svgscale\undefined \makeatother \begin{picture}(1,0.75806896)\lineheight{1}\setlength\tabcolsep{0pt}\put(0,0){\includegraphics[width=\unitlength,page=1]{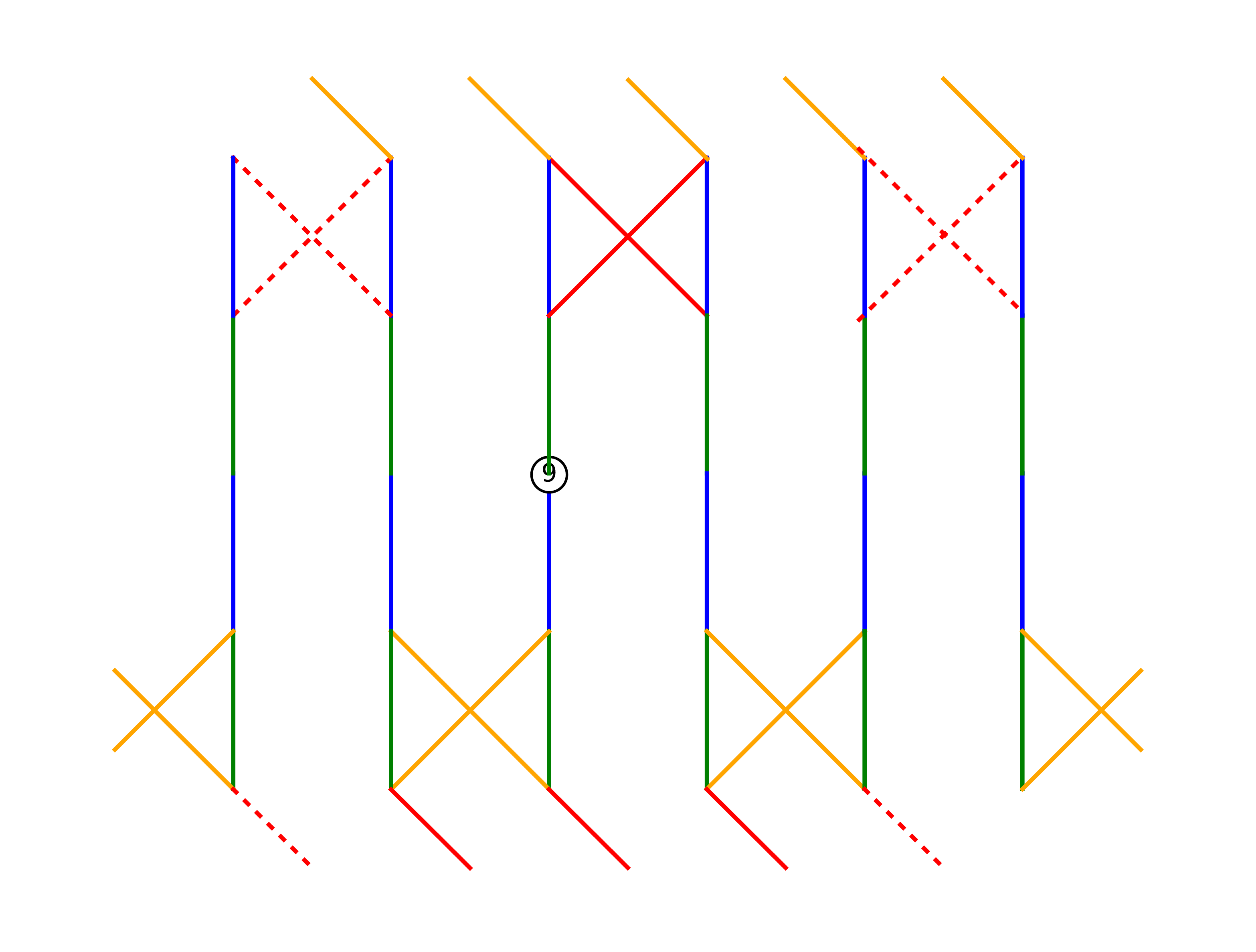}}\put(-2.46945547,-0.35744082){\color[rgb]{0,0,0}\makebox(0,0)[t]{\lineheight{1.25}\smash{\begin{tabular}[t]{c}$X_{(3,5,3,5)}^{\ell}$\end{tabular}}}}\put(-2.79939496,-0.35744082){\color[rgb]{0,0,0}\makebox(0,0)[t]{\lineheight{1.25}\smash{\begin{tabular}[t]{c}$X_{(3,5,3,5)}^{\ell-1}$\end{tabular}}}}\put(-2.12380458,-0.35744082){\color[rgb]{0,0,0}\makebox(0,0)[t]{\lineheight{1.25}\smash{\begin{tabular}[t]{c}$X_{(3,5,3,5)}^{\ell+1}$\end{tabular}}}}\put(0,0){\includegraphics[width=\unitlength,page=2]{3535.pdf}}\put(0.5,0.00392159){\color[rgb]{0,0,0}\makebox(0,0)[t]{\lineheight{1.25}\smash{\begin{tabular}[t]{c}$X_{(3,5,3,5)}^{\ell}$\end{tabular}}}}\put(0.21719468,0.00392159){\color[rgb]{0,0,0}\makebox(0,0)[rt]{\lineheight{1.25}\smash{\begin{tabular}[t]{r}$X_{(3,5,3,5)}^{\ell-1}$\end{tabular}}}}\put(0.78280532,0.00392159){\color[rgb]{0,0,0}\makebox(0,0)[lt]{\lineheight{1.25}\smash{\begin{tabular}[t]{l}$X_{(3,5,3,5)}^{\ell+1}$\end{tabular}}}}\put(0,0){\includegraphics[width=\unitlength,page=3]{3535.pdf}}\end{picture}\endgroup  \caption{The graph $\cX_{(3,5,3,5)}^{t}$.  }
\label{fig:3535}
\end{figure}

The graph $\cX_{(3,5,3,5)}^{t}$ shown in \cref{fig:3535} is the CPR-graph of a family of regular hypertopes with Coxeter diagram 
\begin{equation}\label{eq:cox3535}
\begin{tikzcd}[column sep=small]
\overset{\rho_{0}} {\textcolor{Red}{\bullet}} \arrow[d, dash] \arrow[r, dash, "5"]
& \overset{\rho_{3}}{\textcolor{Orange}{\bullet}} \arrow[d, dash] \\
\underset{\rho_{1}}{\textcolor{Green}\bullet} \arrow[r, dash, "5"']
& \underset{\rho_{2}}{\textcolor{Blue}{\bullet}}
\end{tikzcd}
\end{equation}

As in previous sections, the relations implied by the diagram above for the induced permutation group $G$ follow from \cref{lem:polygonalAction}. 
Note that in \cref{fig:3535} we identify the upper part of the figure with the lower one in such a way that the graph is drawn on the surface of a torus and not on a cylinder as in previous sections.
Also observe that the half-turn at the vertex $(9,0)$ is compatible with the identification and induces an automorphism $\varphi$ of the graph that swaps the colour $0$ with the colour $3$, and the colour $1$ with the colour $2$. 
This graph-automorphism induces a group-automorphism  $\bar{\varphi}: G\to G$ given by swapping the corresponding generators according to the vertical reflection of the diagram in \cref{eq:cox3535}.

It is easy to see that the period of Coxeter elements of $G_{0}$, $G_{1}$, $G_{2}$ and $G_{3}$ is $10$, implying that each of these subgroups are isomorphic to the Coxeter group $[5,3]$.
We use \cref{prop:IP} to prove that the group $G$ has the intersection property. 
To do so, we need to show that $G_{i} \cap G_{j} = G_{i,j}$ for every pair $\left\{ i,j \right\} \subset \left\{ 0,1,2,3\right\}  $.

For each pair $\left\{ i,j \right\} $ we use a similar argument as that in \cref{sec:3434}.
Let $x$ be the vertex $(1,0)$. 
Note that $\left| xG_{0} \right| = \left| x G_{3} \right| = 12 $, which implies that $ \left| \stab_{G_{0}}(x) \right| = \left| \stab_{G_{3}}(x) \right| = 10$. 
It follows that $\stab_{G_{0}} (x) = \stab_{G_{3}}(x) = \left\langle \rho_{1}, \rho_{2} \right\rangle$. 
Observe that $\stab_{G_{0}}(x) \cap \stab_{G_{1}}(x)$ is a subgroup of $\left\langle \rho_{1}, \rho_{2} \right\rangle $ containing $\rho_{2}$. 
Since $\left\langle \rho_{2} \right\rangle $  is a maximal subgroup in $\left\langle \rho_{1}, \rho_{2} \right\rangle $ and $\rho_{1} \not \in G_{1}$, then \[\stab_{G_{0}}(x) \cap \stab_{G_{1}}(x) = \left\langle \rho_{2}  \right\rangle. \]
Observe that
\[\begin{aligned}
&\begin{multlined}
xG_{0} = \{ (1, 0), (2, 0), (3, 0), (4, 0), (5, 0), (6, 0), (7, 0),\\ (8, 0), (9, 0), (10, 0), (11, 0), (12, 0)\}
\end{multlined} && \text{and} \\ 
&\begin{multlined}
xG_{1} = \{ (1, 0), (2, 0), (3, 0), (16, -1), (17, -1), (18, -1), (19, -1), \\ (20, -1), (21, -1), (22, -1), (23, -1), (24, -1)\}.
\end{multlined}
\end{aligned}\]
Since $\left| xG_{0} \cap xG_{1} \right| = 3 $ and $\left| \stab_{G_{0}}(x) \cap \stab_{G_{2}}(x) \right| = 2 $ it follows from \cref{eq:IP_CPR} that \[G_{0} \cap G_{1} = G_{0,1}.\]

Analogously, we can see that 
\[
\stab_{G_{0}}(x) \cap \stab_{G_{2}}(x) = \left\langle \rho_{1} \right\rangle. 
\] 

Note that $xG_{0} \cap xG_{2} = \left\{ (1,0), (2,0), (3,0), (4,0)  \right\} $. We will show that $x(G_{0} \cap G_{2})$ is actually $ \left\{ (1,0), (2,0)\right\} $. 
Observe that \cref{eq:IP_CPR} implies that $|G_{0} \cap G_{2}| \leq 8$. 
In particular, the group $G_{0,2}$ is a normal subgroup in $G_{0} \cap G_{2}$. 
Assume that there exists $\gamma \in G_{0} \cap G_{2}$ such that $(1,0) \gamma = (3,0)$. 
Consider the element $\gamma \rho_{1} \gamma^{-1}$. 
On one hand, $\gamma \rho_{1} \gamma^{-1} \in G_{0,2}$, which implies that $(1,0)\gamma \rho_{1} \gamma^{-1} \in \left\{ (1,0),(2,0) \right\} $.
On the other hand, 
\[(1,0)\gamma \rho_{1} \gamma^{-1} = (3,0) \rho_{1} \gamma^{-1} = (4,0)\gamma^{-1}.\]
It follows that $(2,0)\gamma = (4,0)$.
However, the elements in $G_{0}$ that map $(1,0)$ to $(3,0)$ are the elements in the coset $\left\langle \rho_{1},\rho_{2} \right\rangle \rho_{3}\rho_{2} $ and it can be seen that none of the coset elements map $(2,0)$ to $(4,0)$. 
Therefore $(3,0) \not\in x(G_{0} \cap G_{2}) $
This implies that $x(G_{0} \cap G_{2}) = \left\{ (1,0),(2,0) \right\} $ and from \cref{eq:IP_CPR} it follows that \[G_{0} \cap G_{2} = G_{0,2}.\]

Now let $x=(6,0)$. 
Observe that $\stab_{G_{0}}(x) = \left\langle \rho_{2}, \rho_{1}\rho_{2}\rho_{3}\rho_{2}\rho_{1}\right\rangle $. 
In a similar way as before, $\rho_{1}\rho_{2}\rho_{3}\rho_{2}\rho_{1} \not \in G_{3}$, which implies that \[\stab_{G_{0}}(x) \cap \stab_{G_{3}}(x) = \left\langle \rho_{2} \right\rangle.  \]
The condition \[G_{0} \cap G_{3} = G_{0,3}\] follows directly from \cref{eq:IP_CPR} by observing that  \[xG_{0} \cap xG_{3} = \left\{ (6, 0), (8, 0), (9, 0), (10, 0), (11, 0) \right\}. \]

Notice that for the vertex $x = (4,0)$ we have that $\stab_{G_{1}}(x) = \stab_{G_{2}} = \left\langle \rho_{1},\rho_{2} \right\rangle$.
Since $xG_{1} \cap xG_{2} = \left\{ (4,0) \right\} $, then \cref{eq:IP_CPR} implies that \[G_{1} \cap G_{2} = G_{1,2}.\]

The intersection property for the pairs $\left\{ 1,3 \right\} $ and $\left\{ 2,3 \right\} $ follows from the group automorphism $\bar{\varphi}$ described before, which maps $\rho_{0}$ to $\rho_{3}$ while swapping $\rho_{1}$ and $\rho_{2}$.

To prove that the group $G$ is flag transitive we use \cref{thm:FT_rank3} and a similar argument to that in \cref{sec:3435}. 
We want to prove that \[G_{i} \cap G_{j}G_{k} = (G_{i} \cap G_{j})(G_{i} \cap G_{k}).\]
Note that one inclusion is obvious. We show the other inclusion in detail for $i=1$, $j=0$ and $k=3$. 
The remaining cases follow an analogous idea.

We need to prove that \[G_{1} \cap G_{0} G_{3} \subset G_{1,0}G_{1,3}.\]
As in \cref{sec:3435}, note that 
\[G_{1} \cap G_{0} G_{3} = G_{1} \cap \left( \bigcup_{\alpha \in R} \alpha G_{3} \right) = \bigcup_{\alpha \in R} \left( G_{1} \cap \alpha G_{3} \right),\] 
where $R \subset G_{0}$ is a set of coset representatives of $G_{0,3}$. 
Since $G_{0}$ is isomorphic to the Coxeter group $[5,3]$ and $G_{0,3}$ is one of its maximal parabolic subgroups (of index $12$), the set $R$ can be chosen as 
\[\begin{multlined}
R = \left\{  
\id,
\rho_{3},
\rho_{2}\rho_{3},
\rho_{1}\rho_{2}\rho_{3},
\rho_{2}\rho_{1}\rho_{2}\rho_{3},
\rho_{3}\rho_{2}\rho_{1}\rho_{2}\rho_{3},
\rho_{1}\rho_{2}\rho_{1}\rho_{2}\rho_{3},
\rho_{3}\rho_{1}\rho_{2}\rho_{1}\rho_{2}\rho_{3},
\right. \\ \left. 
\rho_{2}\rho_{3}\rho_{1}\rho_{2}\rho_{1}\rho_{2}\rho_{3},
\rho_{1}\rho_{2}\rho_{3}\rho_{1}\rho_{2}\rho_{1}\rho_{2}\rho_{3},
\rho_{2}\rho_{1}\rho_{2}\rho_{3}\rho_{1}\rho_{2}\rho_{1}\rho_{2}\rho_{3},
\rho_{3}\rho_{2}\rho_{1}\rho_{2}\rho_{3}\rho_{1}\rho_{2}\rho_{1}\rho_{2}\rho_{3},
\right\}. 
\end{multlined}
\]

Since $G$ satisfies the intersection property, then $G_{1} \cap \id G_{3} = G_{1,3}$.
Because $\rho_{3} \in G_{1}$, then $G_{1} \cap \rho_{3} G_{3} = \rho_{3}\left( G_{1} \cap G_{3} \right) = \rho_{3}(G_{1,3})$.
Similarly, $G_{1} \cap \rho_{2}\rho_{3} G_{3} = \rho_{2}\rho_{3}\left( G_{1} \cap G_{3} \right) = \rho_{2}\rho_{3}(G_{1,3})$.

Observe that $\left| G_{1,0}G_{1,3} \right| = \frac{\left| G_{1,0} \right| \cdot \left| G_{1,3} \right|  }{\left( G_{1,0} \cap G_{1,3} \right)} = \frac{6 \cdot 4}{2} = 12 $. 
Since $\left| G_{1,3} \right| = 4 $, we have $\left| G_{1,3} \cup \rho_{3} G_{1,3} \cup \rho_{2}\rho_{3} G_{1,3} \right| = 12$.
It follows that we need to show that $G_{1} \cap \alpha G_{3} = \emptyset$ whenever $\alpha \in R \sm \left\{ \id, \rho_{3}, \rho_{2}\rho_{3} \right\} $. 
Moreover, 	if $\beta \in G_{1}$, we have \[G_{1} \cap \beta \alpha G_{3} = \beta \left( G_{1} \cap \alpha G_{3} \right).\] 
Therefore, it is sufficient to show that \[G_{1} \cap \alpha G_{3} = \emptyset\] for $\alpha \in \left\{ 
\rho_{1}\rho_{2}\rho_{3}, 
\rho_{1}\rho_{2}\rho_{1}\rho_{2}\rho_{3},
\rho_{1}\rho_{2}\rho_{3}\rho_{1}\rho_{2}\rho_{1}\rho_{2}\rho_{3}\right\} $.
To do so, for each of the group elements $\alpha$ we find a vertex $x$ such that the orbit of $x$ under $G_{1}$ and the orbit of $x \alpha$ under $G_{3}$ are disjoint. 

For $\alpha = \rho_{1}\rho_{2}\rho_{3}$ take the vertex $x = (3,0)$; the vertex $x \alpha$ is $(6,0)$. 
Observe that
\[\begin{multlined}
xG_{1} = \{ (1, 0), (2, 0), (3, 0), (16, -1), (17, -1), (18, -1), (19, -1),\\ (20, -1), (21, -1), (22, -1), (23, -1), (24, -1) \},
\end{multlined}
\]
and
\[\begin{multlined}
(x\alpha) G_{3} = \{ (6, 0), (8, 0), (9, 0), (10, 0), (11, 0), (13, 0), (14, 0), \\ (15, 0), (16, 0), (17, 0), (19, 0), (24, 0)\},
\end{multlined}
\] 
Which are disjoint as long as $t \geq 2$.

If $\alpha = \rho_{1}\rho_{2}\rho_{1}\rho_{2}\rho_{3}$ consider the vertex $x = (2,0)$ and for $\alpha = \rho_{1}\rho_{2}\rho_{3}\rho_{1}\rho_{2}\rho_{1}\rho_{2}\rho_{3}$ use the vertex $x = (1,0)$. 
In both cases the vertex $x\alpha $ is $ (6,0)$ and the corresponding orbits are those described above.
This shows that \[G_{1} \cap G_{0} G_{3} = (G_{1,0})(G_{1,3}).\]

We can do a similar analysis to show that \[G_{0} \cap G_{1}G_{2} = (G_{0,1})(G_{0,2}).\]
In other words, we can find a set of coset representatives $R\subset G_{1}$ of $G_{1,2}$  and prove that 
\[G_{0} \cap G_{1} G_{2} = G_{0} \cap \left( \bigcup_{\alpha \in R} \alpha G_{2} \right) = \bigcup_{\alpha \in R} \left( G_{0} \cap \alpha G_{2} \right).\]

Following an analogous argument to that of the previous case, the condition above reduces to show that $G_{0} \cap \alpha G_{2} = \emptyset$ for $\alpha \in \left\{ 
\rho_{0}\rho_{3}\rho_{2}, 
\rho_{0}\rho_{3}\rho_{0}\rho_{3}\rho_{2},
\rho_{0}\rho_{3}\rho_{2}\rho_{0}\rho_{3}\rho_{0}\rho_{3}\rho_{2}\right\} $.

For $\alpha = \rho_{0}\rho_{3}\rho_{2}$ we just need to consider the vertex $x=(1,0)$.
Then we have $x \alpha = (16,-1)$ and 
\[\begin{multlined}
xG_{0} = \{ 1, 0), (2, 0), (3, 0), (4, 0), (5, 0), (6, 0), (7, 0),  \\ (8, 0), (9, 0), (10, 0), (11, 0), (12, 0) \},
\end{multlined}
\]
and
\[\begin{multlined}
(x\alpha) G_{2} = \{ (5, -1), (6, -1), (7, -1), (8, -1), (9, -1), (10, -1), (11, -1), \\ (12, -1), (13, -1), (14, -1), (15, -1), (16, -1)\},
\end{multlined}
\] 
which are disjoint for $t \geq 2$. 

To show that $G_{0} \cap \alpha G_{2}=\emptyset$ for $\alpha=\rho_{0}\rho_{3}\rho_{0}\rho_{3}\rho_{2}$, we require a slightly more detailed analysis. 
Consider the set of vertices $V=\left\{ (1,0),(2,0), (3,0), (10,0) \right\} $. 
It is easy to verify that for each $x \in V$ \[xG_{0} \cap (x\alpha)G_{2}= W:= \left\{ (0,1), (0,2), (0,3), (0,4) \right\}. \]
This implies that if $\gamma \in G_{0} \cap \alpha G_{2}$, then $\gamma$ maps the set $V$ to the set $W$.

Assume that $\gamma$ permutes the vertices $(1,0)$, $(2,0)$, $(3,0)$.
Note that the subgroup $G_{0,1}=\left\langle \rho_{2}, \rho_{3} \right\rangle $ acts as the symmetric group on those vertices, which implies that there exists $\beta \in G_{0,1}$ such that $\gamma\beta$ is an element in $G_{0}$ that fixes $(1,0)$, $(2,0)$ and $(3,0)$. 
Observe that
\[\begin{aligned} 
\stab_{G_{0}}(1,0)&=\left\langle \rho_{1}, \rho_{2} \right\rangle, \\  
\stab_{G_{0}}(2,0)&=\left\langle \rho_{1}, \rho_{3}\rho_{2}\rho_{3} \right\rangle,  \\
\stab_{G_{0}}(3,0)&=\left\langle \rho_{3}, \rho_{2}\rho_{1}\rho_{2} \right\rangle.
\end{aligned}\] 
In particular, those subgroups have trivial intersection, which implies that $\gamma \in G_{0,1}$. 
However, the orbit of $(10,0)$ under is $G_{0,1}$ is $\left\{ (10,0), (11,0), (12,0) \right\} $. 
Therefore $\gamma$ cannot map $(10,0)$ to $(4,0)$, which is a contradiction. 

The previous discussion shows that if $\gamma \in G_{0} \cap \alpha G_{2}$, then 
\[(4,0)\gamma^{-1}\in\left\{ (1,0),(2,0),(3,.0) \right\}.\]
If $(1,0) \gamma = (4,0)$, then $\gamma = \beta \rho_{3} \rho_{2} \rho_{1}$
for some $\beta \in \left\langle \rho_{1}, \rho_{2} \right\rangle = \stab_{G_{0}}(1,0) $. 
Since $(10,0)\gamma \in \left\{ (1,0),(2,0),(3,0) \right\} $ then $\beta$ is either $\rho_{1}\rho_{2}\rho_{1}$ or $\rho_{1}\rho_{2}\rho_{1}\rho_{2}$. 
In the first case $(2,0)\gamma = (7,0)$ and in the second one $(2,0)\gamma = (8,0)$.
Similarly, if $(2,0) \gamma = (4,0)$, then $\gamma = \beta\rho_{2}\rho_{1}$ for some $\beta \in \rho_{3}\left\langle \rho_{1}, \rho_{2} \right\rangle \rho_{3} $. 
By looking at the action of $\gamma$ on $(10,0)$, it can be seen that the only possibilities for $\beta$ are $\rho_{3} \rho_{1} \rho_{2} \rho_{1} \rho_{3}$ and $\rho_{3} \rho_{1} \rho_{2} \rho_{1} \rho_{2} \rho_{3}$ which can be ruled out by looking at their action on $(1,0)$.
Finally if $(3,0)\gamma = (4,0)$, then $\gamma \in \left( \rho_{2} \rho_{3} \left\langle \rho_{1}, \rho_{2} \right\rangle \rho_{3}\rho_{2}  \right) \rho_{1}$, but none of those elements maps $(10,0)$ to $(1,0)$, $(2,0)$ or $(3,0)$. 

We can proceed in a similar way for $\alpha = \rho_{0}\rho_{3}\rho_{2}\rho_{0}\rho_{3}\rho_{0}\rho_{3}\rho_{2}$. 
In this case the sets $V$ and $W$ are $\left\{ (1,0), (2,0), (3,0), (7,0) \right\} $ and $\left\{ (1,0), (2,0), (3,0), (4,0) \right\} $. 
As before, it can be seen that any element $\gamma \in G_{0} \cap \alpha G_{2}$ must satisfy that $(4,0)\gamma^{-1} \in \left\{ (1,0), (2,0), (3,0) \right\} $ and explore all the possibilities for such elements in $G_{0}$ to verify that none of them maps $V$ to $W$. 

As a consequence of the discussion above, we have \[G_{0}\cap G_{1}G_{2} = (G_{0,1})(G_{0,2}).\]
The conditions 
\[\begin{aligned} 
G_{2} \cap G_{3} G_{0} &= (G_{2,3})(G_{2,0}) && \text{and} \\
G_{3}\cap G_{2}G_{1} &= (G_{3,2})(G_{3,1})
\end{aligned}\] 
follow from the action of the group automorphism $\bar{\varphi}$ that swaps
$\rho_{0}$ with $\rho_{3}$, and $\rho_{1}$ with $\rho_{2}$.

It follows from the discussion above and \cref{thm:FT_rank3}  that if $t\geq 2$, then the group $G$ is flag transitive. 
Flag transitivity for the case $t = 1$ can be easily verified using SageMath. 

\subsection{Type \texorpdfstring{$(3,3,3,3,4) $}{(3,3,3,3,4)}.} \label{sec:33334}

\begin{figure}
\def\svgwidth{\textwidth}
\begingroup \makeatletter \providecommand\color[2][]{\errmessage{(Inkscape) Color is used for the text in Inkscape, but the package 'color.sty' is not loaded}\renewcommand\color[2][]{}}\providecommand\transparent[1]{\errmessage{(Inkscape) Transparency is used (non-zero) for the text in Inkscape, but the package 'transparent.sty' is not loaded}\renewcommand\transparent[1]{}}\providecommand\rotatebox[2]{#2}\newcommand*\fsize{\dimexpr\f@size pt\relax}\newcommand*\lineheight[1]{\fontsize{\fsize}{#1\fsize}\selectfont}\ifx\svgwidth\undefined \setlength{\unitlength}{2550.58374023bp}\ifx\svgscale\undefined \relax \else \setlength{\unitlength}{\unitlength * \real{\svgscale}}\fi \else \setlength{\unitlength}{\svgwidth}\fi \global\let\svgwidth\undefined \global\let\svgscale\undefined \makeatother \begin{picture}(1,0.36089742)\lineheight{1}\setlength\tabcolsep{0pt}\put(0,0){\includegraphics[width=\unitlength,page=1]{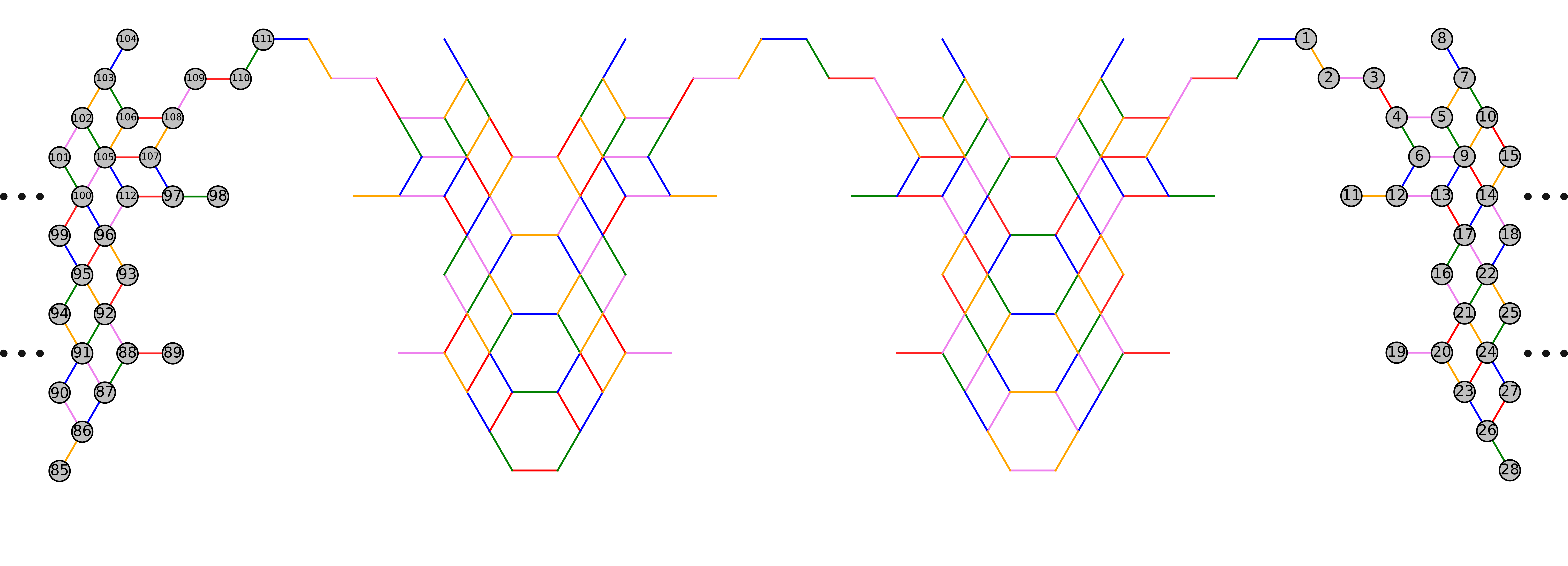}}\put(0.50000001,0.00165711){\color[rgb]{0,0,0}\makebox(0,0)[t]{\lineheight{1.25}\smash{\begin{tabular}[t]{c}$X_{(3,3,3,3,4)}^{\ell}$\end{tabular}}}}\put(0.05244872,0.00443422){\color[rgb]{0,0,0}\makebox(0,0)[t]{\lineheight{1.25}\smash{\begin{tabular}[t]{c}$X_{(3,3,3,3,4)}^{\ell-1}$\end{tabular}}}}\put(0.94755125,0.00443422){\color[rgb]{0,0,0}\makebox(0,0)[t]{\lineheight{1.25}\smash{\begin{tabular}[t]{c}$X_{(3,3,3,3,4)}^{\ell+1}$\end{tabular}}}}\put(0,0){\includegraphics[width=\unitlength,page=2]{33334.pdf}}\end{picture}\endgroup  \caption{The graph $\cX_{(3,3,3,3,4)}^{t}$.  }
\label{fig:33334}
\end{figure}

In \cref{fig:33334} we show the CPR-graph $\cX_{(3,3,3,3,4)}^{t}$ corresponding to the regular hypertopes of type $(3,3,3,3,4)$ and Coxeter diagram as in \cref{eq:cox33334}.
The graph $X_{(3,3,3,3,4)}$ we use to build this graph consist of $112$ vertices so that the vertex set of $\cX_{(3,3,3,3,4)}^{t}$ is $\left\{ 1, \dots, 112 \right\} \times \bZ_{t} $.
The graph $X_{(3,3,3,3,4)}$ admits a symmetry, shown as a vertical reflection in \cref{fig:33334}, that fixes the colour $2$ and swaps the colours $0$ with $4$ and $1$ with $3$.
This symmetry extends to the graph $\cX_{(3,3,3,3,4)}^{t}$, which induces a group automorphism $\bar{\varphi}$ of the associated permutation group $G$.
This group automorphism can be understood in terms of the Coxeter diagram in \cref{eq:cox33334} as the vertical reflection fixing the node labelled with $\rho_{2}$.

\begin{equation}\label{eq:cox33334}
\begin{tikzcd}[column sep=2.3pt, row sep=1.2em, ]
& \overset{\rho_{0}} {\textcolor{Red}{\bullet}} \arrow[ddl, dash] \arrow[rr, dash, "4"] & & \overset{\rho_{4}}{\textcolor{Violet}{\bullet}} \arrow[ddr, dash] &  \\ 
& & & & \\
\underset{\rho_{1}}{\textcolor{Green}\bullet} \arrow[drr, dash]
& & & & \underset{\rho_{3}}{\textcolor{Orange}{\bullet}} \arrow[dll, dash]\\
& &\underset{\rho_{2}}{\textcolor{Blue}{\bullet}} & & 
\end{tikzcd}
\end{equation}

The relations for the permutation group $G$ implicit in the Coxeter diagram are easy to verify using \cref{lem:polygonalAction}. 
The subgroups $G_{0}$ and $G_{4}$ are isomorphic to the Coxeter group $[3,3,3,3]$ of order $120$; the groups $G_{1}$ and $G_{3}$ are isomorphic to $[3,3,4]$ (of order $384$) whereas the subgroup $G_{2}$ is isomorphic to $[3,4,3]$ of order $1152$.
This can be seen as before, by looking at the period of the corresponding Coxeter elements, which are $5$ for $G_{0}$ and $G_{4}$, $8$ for $G_{1}$ and $G_{3}$ and $12$ for $G_{2}$.
Alternatively, one can see that the orbits of certain vertices induce known CPR-graphs for each of these groups. 
For example, the orbit of $(1,0)$ under $G_{2}$ consists of $24$ vertices and induces the CPR-graph for $[3,4,3]$ given by the action of the Coxeter group on the facets of a regular $24$-cell.
It follows that the group $G_{2}$ is isomorphic to $[3,4,3]$ (see \cref{rem:mix}).

Since $G_{i}$ is a C-group for every $i \in \left\{ 0,1,2,3,4 \right\} $, we can use \cref{lem:IP_CPR} to prove that the group $G$ itself satisfies the intersection property. 
In \cref{tab:IP_33334} we specify a vertex $x$ for six of the pairs $\left\{ i,j \right\} $ for which we can use \cref{eq:IP_gcd} to prove that \[G_{i}\cap G_{j} = G_{i,j}.\] In this table $s_{i}$ and $s_{j}$ denote $\left| \stab_{G_{i}}(x) \right| $ and $\left| \stab_{G_{j}}(x)\right| $, respectively. 
The condition for the remaining pairs follows from the action of the group automorphism $\bar{\varphi}$ described above.

\begin{table}
\begin{tabularx}{\textwidth}{| \cc{.6} | \cc{.6} | \cc{.8} | \cc{.6} | \cc{.6} | \cc{1.9} | \cc{1.9} | \cc{1} |}
    \hline
$i$ & $j$ & $x$ & $s_{i}$ & $s_{j}$ & $\gcd(s_{i},s_{j})$ & $|xG_{i}\cap xG_{j}|$ & $G_{i,j}$ \\ \hline
$0$ & $1$ & $(6,0)$ & $12$ & $16$ & $4$ & $6$ & $24$ \\ \hline
$0$ & $2$ & $(8,0)$ & $12$ & $1152$ & $12$ & $1$ & $12$ \\ \hline
$0$ & $3$ & $(11,0)$ & $12$ & $384$ & $12$ & $1$ & $12$ \\ \hline
$0$ & $4$ & $(3,0)$ & $24$ & $24$ & $24$ & $1$ & $24$ \\ \hline
$1$ & $2$ & $(8,0)$ & $48$ & $1152$ & $48$ & $1$ & $48$ \\ \hline
$1$ & $3$ & $(11,0)$ & $16$ & $384$ & $16$ & $1$ & $16$ \\ \hline

\end{tabularx}
\caption{Intersection property for $\cX_{(3,3,3,3,4)}^{t}$}
\label{tab:IP_33334}
\end{table}

To show that the group $G$ induces a regular hypertope it only remain to show that $G$ is flag-transitive.
To do so, we use \cref{lem:FT_CPR}.
More precisely, we show  in \cref{tab:FT_33334} a vertex $x$ for some of the triples $( i,j,k ) $ in such a way that the vertex satisfies \cref{eq:FT_CPR}.
In \cref{tab:IP_33334}, $o_{i,j,k}$ denotes the size of $xG_{i}\cap xG_{j}G_{k}$ and $s_{i} = \left| \stab_{G_{i}}(x) \right| $.
The condition in \cref{eq:FT_CPR} in the notation of \cref{tab:FT_33334} is \[o_{i,j,k} \cdot s_{i} \leq \frac{|G_{i,j}|\cdot |G_{i,k}|}{|G_{i,j,k}|}.\]
In a similar way as with the intersection property, the condition $G_{i}\cap G_{j}G_{k}= G_{i,j}G_{i,k}$ for the subsets $\left\{ i,j,k \right\} $ not listed in \cref{tab:FT_33334} follows from the action of the group automorphism $\bar{\varphi}$.

\begin{table}
\begin{tabularx}{\textwidth}{| \cc{1} | \cc{1} | \cc{1} | \cc{1} | \cc{1} | \cc{1} | \cc{1} | \cc{1} | \cc{1} |}
\hline
$i$ & $j$ & $k$ & $x$ & $o_{i,j,k}$ & $s_{i}$ & $|G_{i,j}|$ & $|G_{i,k}|$ & $|G_{i,j,k}|$\\ \hline 
$0$ & $1$ & $2$ & $(16, 0)$ & $8$ & $6$ & $24$ & $12$ & $6$ \\ \hline 
$0$ & $1$ & $3$ & $(16, 0)$ & $12$ & $6$ & $24$ & $12$ & $4$ \\ \hline 
$0$ & $1$ & $4$ & $(16, 0)$ & $16$ & $6$ & $24$ & $24$ & $6$ \\ \hline 
$0$ & $2$ & $3$ & $(8, 0)$ & $3$ & $12$ & $12$ & $12$ & $4$ \\ \hline 
$0$ & $2$ & $4$ & $(8, 0)$ & $6$ & $12$ & $12$ & $24$ & $4$ \\ \hline 
$1$ & $2$ & $3$ & $(8, 0)$ & $2$ & $48$ & $48$ & $16$ & $8$ \\ \hline 
\end{tabularx}
\caption{Flag transitivity for $\cX_{(3,3,3,3,4)}^{t}$}
\label{tab:FT_33334}
\end{table}

Notably, the arguments used above for the group $G$ induced by the graph $\cX_{(3,3,3,3,4)}^{t}$ hold even when $t=1$.
\cref{thm:CGroupFlagTrans_Hypertope} implies that the permutation group $G$ induced by the graph  $\cX_{(3,3,3,3,4)}^{t}$ is the automorphism group of a regular hypertope with Coxeter diagram as in \cref{eq:cox33334}.

\subsection{Type \texorpdfstring{$ \{ 5, {}_{3}^{3} \} $}{ \{5, 3-3\} }.} \label{sec:5-33}

In this section we the build proper $4$-edge-coloured graphs $\cX_{\{ 5, {}_{3}^{3} \}}^{t}$ corresponding to a family of regular hypertopes with Coxeter diagram as in \cref{eq:cox5-33}.

\begin{equation}\label{eq:cox5-33}
\begin{tikzcd}[column sep=1em, row sep=.3em, ]
 	& & [-0.3em] \overset{\rho_{2}}{\textcolor{Blue}{\bullet}} \arrow[dddl, dash] \\
 	& & \\
 	& & \\
 	\underset{\rho_{0}} {\textcolor{Red}{\bullet}} \arrow[r, dash, "5"]
 	&\underset{\rho_{1}}{\textcolor{Green}\bullet} \arrow[dr, dash] & \\
 	& & [-0.3em] \underset{\rho_{3}}{\textcolor{Orange}{\bullet}} 
\end{tikzcd}
\end{equation}

To build the graph $X_{\{ 5, {}_{3}^{3} \}}$ we follow a slightly different method from the one used in previous sections.
Consider the graph in \cref{fig:5-33_A}.
This graph is a CPR-graph for the Coxeter group $\{5,3\}$. 
In fact, this is the CPR-graph induced by the action of such group on the faces of the dodecahedron. 
The graph $X_{\{ 5, {}_{3}^{3} \}}$ is built as follows. 
Take two copies of the graph in \cref{fig:5-33_A}, one with vertex set $\left\{ 1, \dots, 12 \right\} $ and the other one with vertex set $\left\{ 13, \dots, 24 \right\} $ so that if $i,j\in\left\{ 1,\dots,12 \right\} $ then there is an edge connecting $i$ and $j$ if and only if there is an edge connecting $i+12$ and $j+12$. 
Assume that $i$ and $j$ are connected by an edge marked with a cross (such as the edge connecting $5$ and $6$ in \cref{fig:5-33_A}).
This means that the there is an edge in the graph $X_{\{ 5, {}_{3}^{3} \}}$ (of the same colour as the cross) connecting the vertex $i$ with the vertex $j'=j+12$ and another edge connecting the vertex $j$ with the vertex $i'=i+12$ as in \cref{fig:5-33_B}.
The the graph $X_{\{ 5, {}_{3}^{3} \}}$ is shown in \cref{fig:5-33_C}.

\begin{figure}
	\centering
	\begin{subfigure}[b]{0.75\textwidth}
		\centering
			\begin{scriptsize}
\def\svgwidth{\textwidth}
		\begingroup \makeatletter \providecommand\color[2][]{\errmessage{(Inkscape) Color is used for the text in Inkscape, but the package 'color.sty' is not loaded}\renewcommand\color[2][]{}}\providecommand\transparent[1]{\errmessage{(Inkscape) Transparency is used (non-zero) for the text in Inkscape, but the package 'transparent.sty' is not loaded}\renewcommand\transparent[1]{}}\providecommand\rotatebox[2]{#2}\newcommand*\fsize{\dimexpr\f@size pt\relax}\newcommand*\lineheight[1]{\fontsize{\fsize}{#1\fsize}\selectfont}\ifx\svgwidth\undefined \setlength{\unitlength}{585.09967041bp}\ifx\svgscale\undefined \relax \else \setlength{\unitlength}{\unitlength * \real{\svgscale}}\fi \else \setlength{\unitlength}{\svgwidth}\fi \global\let\svgwidth\undefined \global\let\svgscale\undefined \makeatother \begin{picture}(1,0.17442715)\lineheight{1}\setlength\tabcolsep{0pt}\put(0,0){\includegraphics[width=\unitlength,page=1]{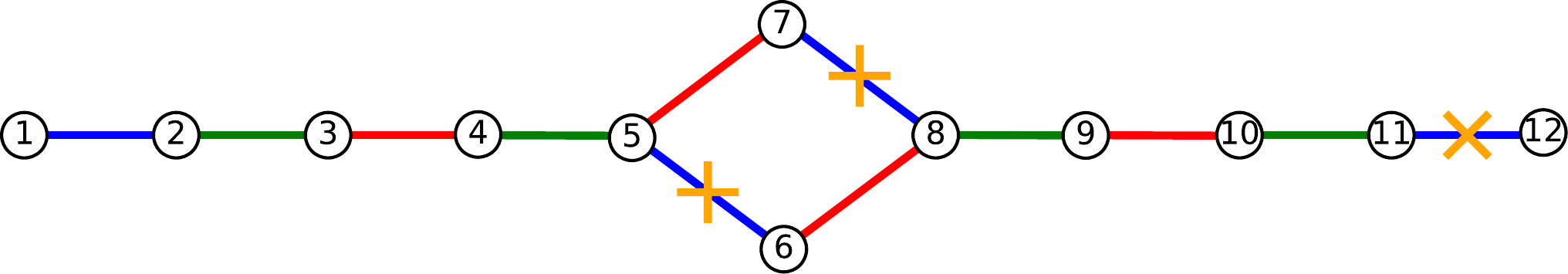}}\end{picture}\endgroup  		\end{scriptsize}
		\caption{Half of the graph $X_{\{ 5, {}_{3}^{3} \}}$}\label{fig:5-33_A}
	\end{subfigure}

	\begin{subfigure}[b]{0.45\textwidth}
		\centering
			\begin{scriptsize}
\def\svgwidth{\textwidth}
		\begingroup \makeatletter \providecommand\color[2][]{\errmessage{(Inkscape) Color is used for the text in Inkscape, but the package 'color.sty' is not loaded}\renewcommand\color[2][]{}}\providecommand\transparent[1]{\errmessage{(Inkscape) Transparency is used (non-zero) for the text in Inkscape, but the package 'transparent.sty' is not loaded}\renewcommand\transparent[1]{}}\providecommand\rotatebox[2]{#2}\newcommand*\fsize{\dimexpr\f@size pt\relax}\newcommand*\lineheight[1]{\fontsize{\fsize}{#1\fsize}\selectfont}\ifx\svgwidth\undefined \setlength{\unitlength}{436.94616699bp}\ifx\svgscale\undefined \relax \else \setlength{\unitlength}{\unitlength * \real{\svgscale}}\fi \else \setlength{\unitlength}{\svgwidth}\fi \global\let\svgwidth\undefined \global\let\svgscale\undefined \makeatother \begin{picture}(1,0.34266687)\lineheight{1}\setlength\tabcolsep{0pt}\put(0,0){\includegraphics[width=\unitlength,page=1]{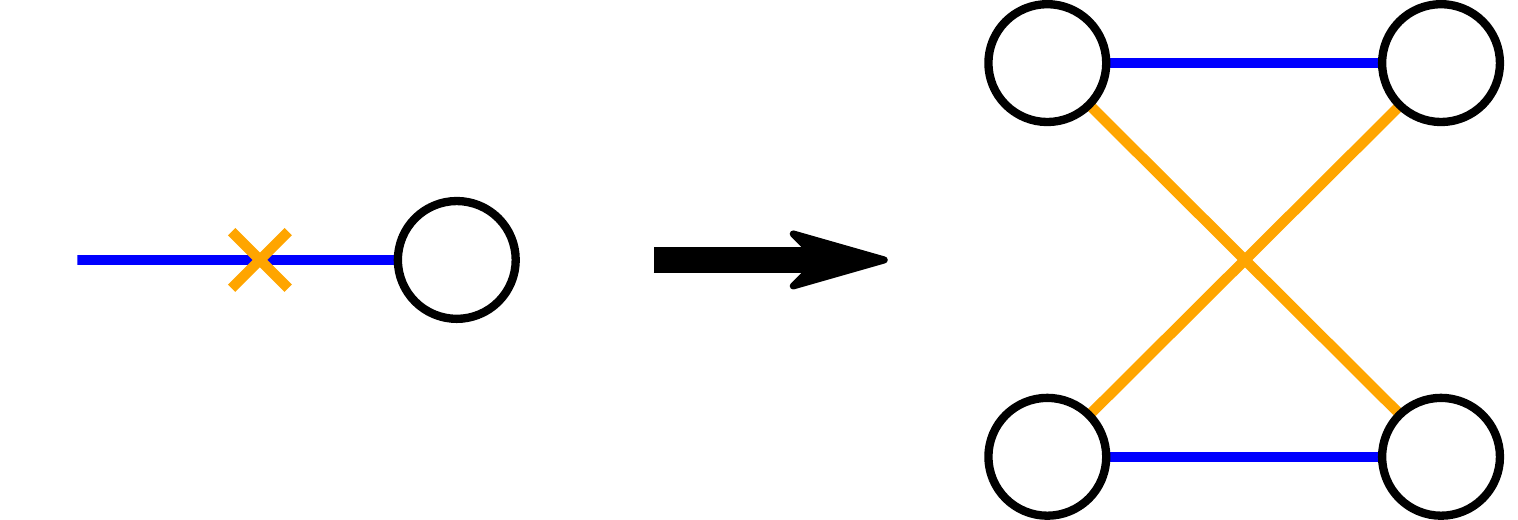}}\put(0.69260531,0.29056303){\color[rgb]{0,0,0}\makebox(0,0)[t]{\lineheight{1.25}\smash{\begin{tabular}[t]{c}\textit{$i$}\end{tabular}}}}\put(0.9521018,0.29056296){\color[rgb]{0,0,0}\makebox(0,0)[t]{\lineheight{1.25}\smash{\begin{tabular}[t]{c}\textit{$j$}\end{tabular}}}}\put(0.69260553,0.03106699){\color[rgb]{0,0,0}\makebox(0,0)[t]{\lineheight{1.25}\smash{\begin{tabular}[t]{c}\textit{$i'$}\end{tabular}}}}\put(0.95039204,0.03211463){\color[rgb]{0,0,0}\makebox(0,0)[t]{\lineheight{1.25}\smash{\begin{tabular}[t]{c}\textit{$j'$}\end{tabular}}}}\put(0,0){\includegraphics[width=\unitlength,page=2]{5-33_B.pdf}}\put(0.04386499,0.16081512){\color[rgb]{0,0,0}\makebox(0,0)[t]{\lineheight{1.25}\smash{\begin{tabular}[t]{c}\textit{$i$}\end{tabular}}}}\put(0.30336126,0.16081512){\color[rgb]{0,0,0}\makebox(0,0)[t]{\lineheight{1.25}\smash{\begin{tabular}[t]{c}\textit{$j$}\end{tabular}}}}\end{picture}\endgroup  		\end{scriptsize}
		\caption{}\label{fig:5-33_B}
	\end{subfigure}
	
	\begin{subfigure}[b]{0.75\textwidth}
		\centering
			\begin{scriptsize}
\def\svgwidth{\textwidth}
		\begingroup \makeatletter \providecommand\color[2][]{\errmessage{(Inkscape) Color is used for the text in Inkscape, but the package 'color.sty' is not loaded}\renewcommand\color[2][]{}}\providecommand\transparent[1]{\errmessage{(Inkscape) Transparency is used (non-zero) for the text in Inkscape, but the package 'transparent.sty' is not loaded}\renewcommand\transparent[1]{}}\providecommand\rotatebox[2]{#2}\newcommand*\fsize{\dimexpr\f@size pt\relax}\newcommand*\lineheight[1]{\fontsize{\fsize}{#1\fsize}\selectfont}\ifx\svgwidth\undefined \setlength{\unitlength}{585.20471191bp}\ifx\svgscale\undefined \relax \else \setlength{\unitlength}{\unitlength * \real{\svgscale}}\fi \else \setlength{\unitlength}{\svgwidth}\fi \global\let\svgwidth\undefined \global\let\svgscale\undefined \makeatother \begin{picture}(1,0.36830293)\lineheight{1}\setlength\tabcolsep{0pt}\put(0,0){\includegraphics[width=\unitlength,page=1]{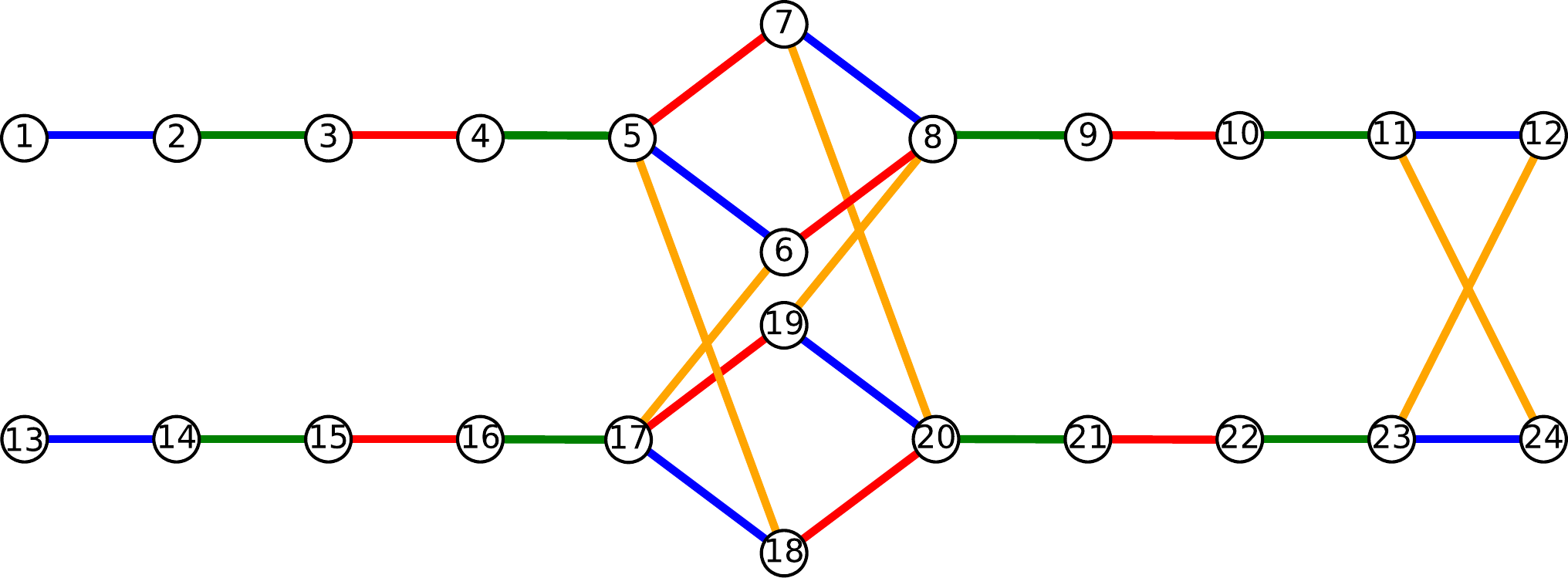}}\end{picture}\endgroup  		\end{scriptsize}
		\caption{The graph $X_{\{ 5, {}_{3}^{3} \}}$}\label{fig:5-33_C}
	\end{subfigure}
	\caption{}
\end{figure}

Finally, to build the graph $\cX_{\{ 5, {}_{3}^{3} \}}^{t}$ we just proceed as in previous sections by taking as vertex set the set $\left\{ 1, \dots, 24 \right\}\times \bZ_{t} $ and the (solid) edges between the vertices $(i, \ell)$ and  $(j,\ell)$ whenever $i$ and $j$ are connected in $X_{\{ 5, {}_{3}^{3} \}}$ together with the (dotted) edges connecting $(13,\ell)$ with $(2,\ell + 1)$, and $(14,\ell)$ with $(1, \ell +1)$ for every $\ell \in \bZ_{t}$. 
This graph $\cX_{\{ 5, {}_{3}^{3} \}}^{t}$ is shown in \cref{fig:5-33}.

\begin{figure}
\def\svgwidth{\textwidth}
\begingroup \makeatletter \providecommand\color[2][]{\errmessage{(Inkscape) Color is used for the text in Inkscape, but the package 'color.sty' is not loaded}\renewcommand\color[2][]{}}\providecommand\transparent[1]{\errmessage{(Inkscape) Transparency is used (non-zero) for the text in Inkscape, but the package 'transparent.sty' is not loaded}\renewcommand\transparent[1]{}}\providecommand\rotatebox[2]{#2}\newcommand*\fsize{\dimexpr\f@size pt\relax}\newcommand*\lineheight[1]{\fontsize{\fsize}{#1\fsize}\selectfont}\ifx\svgwidth\undefined \setlength{\unitlength}{1241.40881348bp}\ifx\svgscale\undefined \relax \else \setlength{\unitlength}{\unitlength * \real{\svgscale}}\fi \else \setlength{\unitlength}{\svgwidth}\fi \global\let\svgwidth\undefined \global\let\svgscale\undefined \makeatother \begin{picture}(1,0.43907069)\lineheight{1}\setlength\tabcolsep{0pt}\put(0,0){\includegraphics[width=\unitlength,page=1]{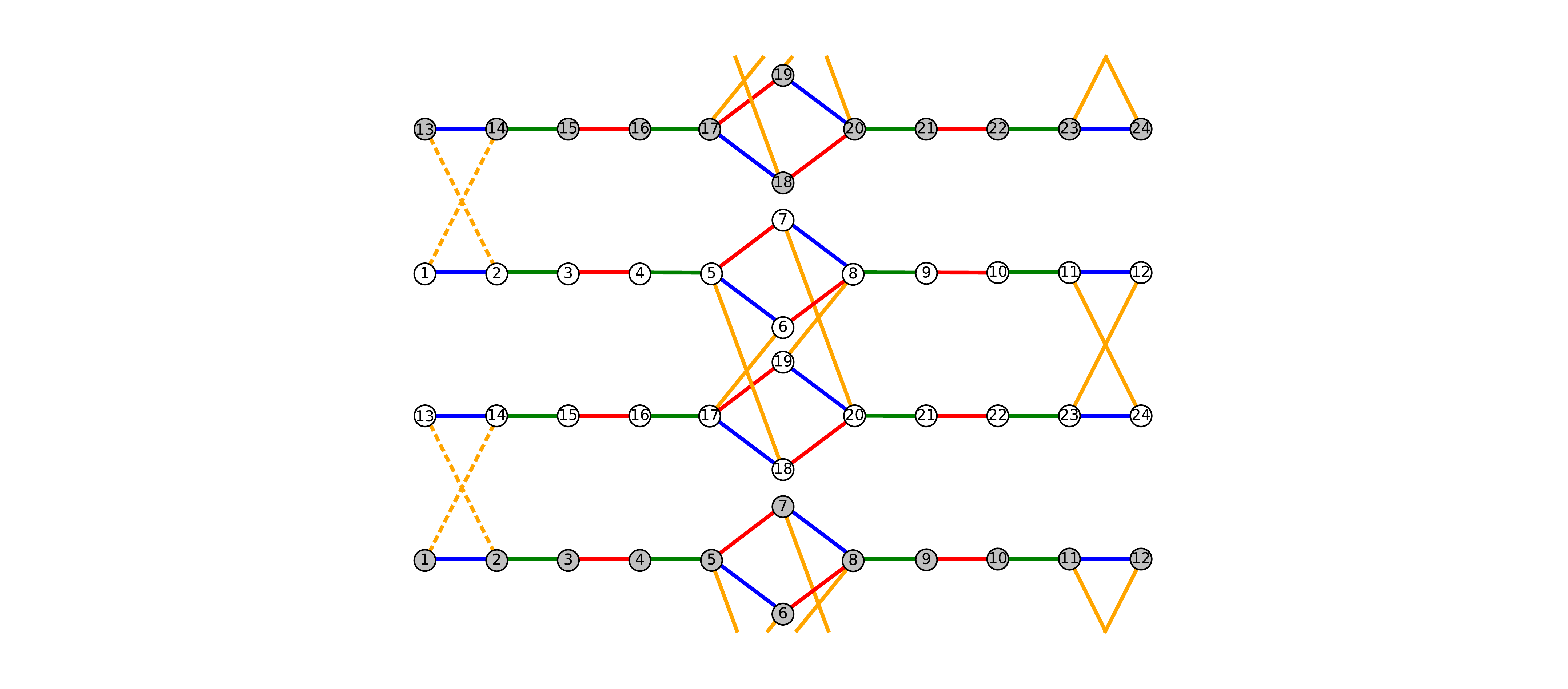}}\put(1.63028812,-0.46528969){\color[rgb]{0,0,0}\makebox(0,0)[t]{\lineheight{1.25}\smash{\begin{tabular}[t]{c}$X_{(3,5,3,5)}^{\ell}$\end{tabular}}}}\put(1.53895173,-0.46451837){\color[rgb]{0,0,0}\makebox(0,0)[rt]{\lineheight{1.25}\smash{\begin{tabular}[t]{r}$X_{(3,5,3,5)}^{\ell-1}$\end{tabular}}}}\put(1.72162456,-0.46451837){\color[rgb]{0,0,0}\makebox(0,0)[lt]{\lineheight{1.25}\smash{\begin{tabular}[t]{l}$X_{(3,5,3,5)}^{\ell+1}$\end{tabular}}}}\put(0.22599072,0.35673779){\color[rgb]{0,0,0}\makebox(0,0)[rt]{\lineheight{1.25}\smash{\begin{tabular}[t]{r}$X_{\{5, {}_{3}^{3}\} }^{\ell-1}$\end{tabular}}}}\put(0,0){\includegraphics[width=\unitlength,page=2]{5-33_D.pdf}}\put(0.22599072,0.08272855){\color[rgb]{0,0,0}\makebox(0,0)[rt]{\lineheight{1.25}\smash{\begin{tabular}[t]{r}$X_{\{5, {}_{3}^{3} \}}^{\ell+1}$\end{tabular}}}}\put(0.22737219,0.21658166){\color[rgb]{0,0,0}\makebox(0,0)[rt]{\lineheight{1.25}\smash{\begin{tabular}[t]{r}$X_{\{5, {}_{3}^{3}\} }^{\ell}$\end{tabular}}}}\put(0,0){\includegraphics[width=\unitlength,page=3]{5-33_D.pdf}}\end{picture}\endgroup  \caption{The graph $\cX_{\{ 5, {}_{3}^{3} \}}^{t}$.  }
\label{fig:5-33}
\end{figure}

Now consider the induced permutation group $G$. 
Observe that all $\left\{ 0,1,2 \right\}$-componentes of the graph $\cX_{\{ 5, {}_{3}^{3} \}}$ are isomorphic to the graph shown in \cref{fig:5-33_A}. 
This implies that the subgroup $G_{3}$ is actually isomorphic to the Coxeter group $[5,3]$. 
Note that the $\left\{ 1,2,3 \right\}$-components of the resulting graph are isomorphic. 
More precisely, all of them are the CPR-graph induced by the action of the Coxeter group $[3,3]$ on the edges of the tetrahedron.
This implies that the subgroup $G_{0}$ is isomorphic to the group $[3,3]$.
Finally, note that the $\left\{ 2,3 \right\}$-components of $\cX_{\{ 5, {}_{3}^{3} \}}$  are either alternating squares of isolated vertices. 
\cref{lem:polygonalAction} implies that the permutations $\rho_{2}$ and $\rho_{3}$ commute.

Consider the involution $\varphi$ of the vertex set of $\cX_{\{ 5, {}_{3}^{3} \}}$ given by 
\[\begin{aligned} 
	(1,\ell) &\leftrightarrow (13,\ell-1), \\ 
	(6,\ell) &\leftrightarrow (18,\ell),  \\
	(7,\ell) &\leftrightarrow (19,\ell),  \\
	(8,\ell) &\leftrightarrow (20,\ell),  \\
	(9,\ell) &\leftrightarrow (21,\ell),  \\
	(10,\ell) &\leftrightarrow (22,\ell), \\
	(11,\ell) &\leftrightarrow (23,\ell), \\
	(12,\ell) &\leftrightarrow (23,\ell), 
\end{aligned}\] 
for every $\ell \in \bZ_{t}$, while fixing all the other vertices.
This involution induces a graph-automorphism that swaps the edges of colour $2$ with the edges of colour $3$, which in turn induces a group automorphism $\bar{\varphi}: G \to G$ that swaps the generators $\rho_{2}$ and $\rho_{3}$, and fixes $\rho_{0}$ and $\rho_{1}$.
This automorphism can be seen as horizontal reflection in the diagram in \cref{eq:cox5-33}. 

The automorphism $\bar{\varphi}$ maps the subgroup $G_{2}$ to the subgroup $G_{3}$, which implies that the latter is also isomorphic to the Coxeter group $[5,3]$.
The discussion above shows that not only the group $G$ satisfies the relations implicit by the Coxeter diagram in \cref{eq:cox5-33}, but also that the subgroups $G_{i}$, $i \in \left\{ 0,1,2,3 \right\} $ satisfy the intersection property.
As before, we use \cref{prop:IP} to show that $G$ satisfies the intersection property. 
Therefore, we only need to show that for every pair $i,j \in \left\{ 0,1,2,3 \right\} $ the equality $G_{i} \cap G_{j}=G_{i,j}$ holds. 

According to \cref{lem:IP_CPR} we only need to find a vertex $x$ for each pair $i,j$ that satisfies that \[\left| xG_{i} \cap xG_{j} \right|\cdot \gcd\left( |\stab_{G_{i}}(x)|,|\stab_{G_{j}}(x)| \right)\leq \left| G_{i,j} \right|.  \]
The corresponding vertices are listed below and the inequality above can be easily checked from the graph.
\begin{itemize}
	\item $x=(3,0)$ for $\left\{ i,j \right\} =\left\{ 0,1 \right\} $, 
	\item $x=(1,0)$ for $\left\{ i,j \right\} =\left\{ 0,2 \right\} $, 
	\item $x=(1,0)$ for $\left\{ i,j \right\} =\left\{ 1,2 \right\} $, 
	\item $x=(1,0)$ for $\left\{ i,j \right\} =\left\{ 2,3 \right\} $. 
\end{itemize}
The condition for the pairs $\left\{ 0,3 \right\} $ and $\left\{ 1,3 \right\} $ follows from the action of the group automorphism $\bar{\varphi}$.

Now we need to prove that the group $G$ is flag transitive. 
According to \cref{thm:FT_rank3} and \cref{lem:FT_Group} we just need to prove that \[G_{i} \cap G_{j}G_{k} = (G_{i,j})(G_{i,k})\] for every subset $\left\{ i,j,k \right\} \subset \left\{ 0,1,2,3 \right\} $.
As before, we can use \cref{lem:FT_CPR} to prove the condition for some of the subsets $\left\{ i,j,k \right\} $. 
Given $i$, $j$ and $k$, we need a vertex $x$ such that 
\begin{equation}\label{eq:FT_5-33}
\left|\left(  x G_i \cap x G_jG_k \right) \right|  \left|\stab_{G_i}(x)\right| \leq \left| (G_{i,j})(G_{i,k}) \right| = \frac{|G_{i,j}|\cdot|G_{i,k}|}{2} 
\end{equation}

We list the corresponding vertices for three of the four possible subsets. 
\begin{itemize}
	\item $x=(3,0)$ if $(i,j,k)=(0,1,2)$,
	\item $x=(3,0)$ if $(i,j,k)=(0,1,3)$,
	\item $x=(1,0)$ if $(i,j,k)=(1,2,3)$.
\end{itemize}

If $\left\{ i,j,k \right\} = \{0,2,3\}$ then we cannot find such vertex $x$.
We will show that \[G_{0} \cap G_{2}G_{3} = (G_{0,2})(G_{0,3})\] using the technique described in detail in \cref{sec:3435}.
Observe that \[G_{0}\cap G_{2}G_{3} = \bigcup_{\alpha \in R}\left( G_{0} \cap \alpha G_{3} \right),\] where $R$ is a set of coset representatives of $G_{2,3}$ in $G_{2}$. In particular $R$ can be taken as 
\[
  \begin{multlined}
\{ 
\id,
\rho_{3},
\rho_{1}\rho_{3},
\rho_{0}\rho_{1}\rho_{3},
\rho_{1}\rho_{0}\rho_{1}\rho_{3},
\rho_{3}\rho_{1}\rho_{0}\rho_{1}\rho_{3},
\rho_{0}\rho_{1}\rho_{0}\rho_{1}\rho_{3}, 
\rho_{3}\rho_{0}\rho_{1}\rho_{0}\rho_{1}\rho_{3},
\\
\rho_{1}\rho_{3}\rho_{0}\rho_{1}\rho_{0}\rho_{1}\rho_{3},
\rho_{0}\rho_{1}\rho_{3}\rho_{0}\rho_{1}\rho_{0}\rho_{1}\rho_{3},
\rho_{1}\rho_{0}\rho_{1}\rho_{3}\rho_{0}\rho_{1}\rho_{0}\rho_{1}\rho_{3},
\rho_{3}\rho_{1}\rho_{0}\rho_{1}\rho_{3}\rho_{0}\rho_{1}\rho_{0}\rho_{1}\rho_{3} \}. 
\end{multlined}
\]

Observe that \[\left| G_{0} \cap G_{3} \right| = \left| G_{0} \cap \rho_{3}G_{3} \right| = \left| G_{0} \cap \rho_{1}\rho_{3}G_{3} \right| = 6,  \]
Since $\left| (G_{0,2})(G_{0,3}) \right| = 18 $, we need to show that $G_{0} \cap \alpha G_{3} = \emptyset$ for $\alpha \in R \sm \left\{ \id, \rho_{3}, \rho_{1}\rho_{3} \right\} $.
In fact, if $\beta \in G_{0}$, then $\left| G_{0} \cap \alpha G_{3} \right| = \left| G_{0} \cap \beta \alpha G_{3} \right| $. 
Therefore, it is enough to show that $G_{0} \cap \alpha G_{3} = \emptyset$ for $\alpha \in \left\{  
\rho_{0}\rho_{1}\rho_{3},
\rho_{0}\rho_{1}\rho_{0}\rho_{1}\rho_{3}, 
\rho_{0}\rho_{1}\rho_{3}\rho_{0}\rho_{1}\rho_{0}\rho_{1}\rho_{3}
\right\} $. 

To prove that  $G_{0} \cap \alpha G_{3} = \emptyset$ for the mentioned group elements $\alpha$, we exhibit a vertex $x$ such that the orbit of $x$ under $G_{0}$ and the orbit of $x \alpha$ under $G_{3}$ are disjoint. 

The vertices $x$ are listed below.
\begin{itemize}
	\item $x = (3,0)$ if $\alpha = \rho_{0}\rho_{1}\rho_{3}$,
	\item $x = (2,0)$ if $\alpha= \rho_{0}\rho_{1}\rho_{0}\rho_{1}\rho_{3}$,
	\item $x = (13,-1)$ if $\alpha = \rho_{0}\rho_{1}\rho_{3}\rho_{0}\rho_{1}\rho_{0}\rho_{1}\rho_{3}$.
\end{itemize}
For each case, the vertex $x$ satisfies that $x \alpha = (18,0)$.
Moreover, those three vertices belong to the same orbit under $G_{0}$ which is disjoint to the orbit of $(18,0)$ under $G_{3}$ as long as $t\geq 2$.

Flag transitivity and the intersection property for the permutation group induced by $\cX_{\{ 5, {}_{3}^{3} \}}^{1}$ can be easily checked using SageMath.

\subsection{Type \texorpdfstring{$ \{ 5, 3, {}_{3}^{3} \} $}{ \{5, 3, 3-3\} }.} \label{sec:53-33}

\begin{figure}
\def\svgwidth{\textwidth}
\begingroup \makeatletter \providecommand\color[2][]{\errmessage{(Inkscape) Color is used for the text in Inkscape, but the package 'color.sty' is not loaded}\renewcommand\color[2][]{}}\providecommand\transparent[1]{\errmessage{(Inkscape) Transparency is used (non-zero) for the text in Inkscape, but the package 'transparent.sty' is not loaded}\renewcommand\transparent[1]{}}\providecommand\rotatebox[2]{#2}\newcommand*\fsize{\dimexpr\f@size pt\relax}\newcommand*\lineheight[1]{\fontsize{\fsize}{#1\fsize}\selectfont}\ifx\svgwidth\undefined \setlength{\unitlength}{2215.12329102bp}\ifx\svgscale\undefined \relax \else \setlength{\unitlength}{\unitlength * \real{\svgscale}}\fi \else \setlength{\unitlength}{\svgwidth}\fi \global\let\svgwidth\undefined \global\let\svgscale\undefined \makeatother \begin{picture}(1,0.22177641)\lineheight{1}\setlength\tabcolsep{0pt}\put(0,0){\includegraphics[width=\unitlength,page=1]{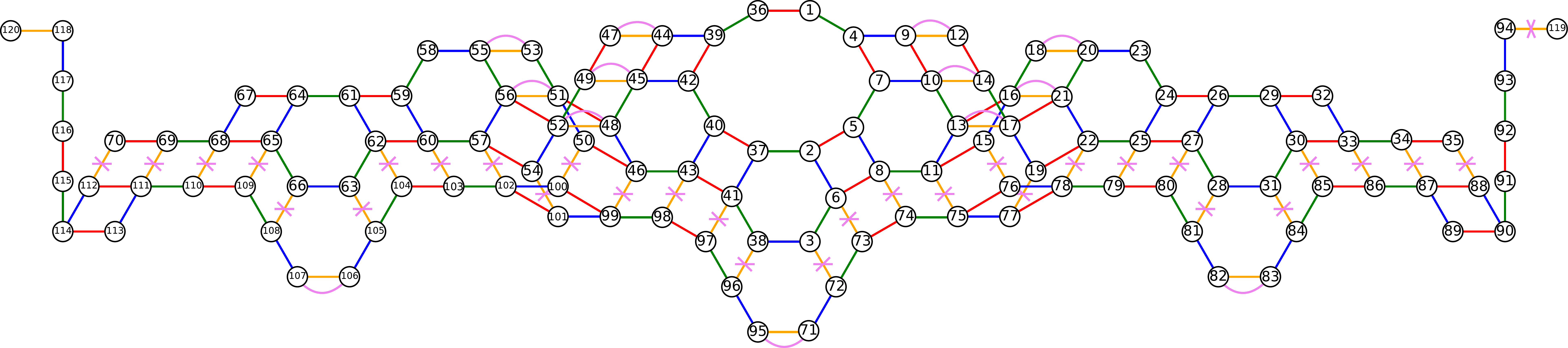}}\end{picture}\endgroup  \caption{Half of the graph $X_{\{ 5, 3, {}_{3}^{3} \}}$.  }
\label{fig:53-33_A}
\end{figure}

The graphs $\cX^{t}_{ \{ 5, 3, {}_{3}^{3} \}}$ associated with the regular hypertopes of type $ \{ 5, 3, {}_{3}^{3} \}$ are built from the graph in \cref{fig:53-33_A} following a similar idea as the one used in the previous section.
The graph in \cref{fig:5-33_A} has $120$ vertices and if we remove the edges of colour $4$ it is in fact the CPR-graph induced by the action of the Coxeter group $[5,3,3]$ on the cells of a regular $120$-cell.
The graph $X_{ \{ 5, 3, {}_{3}^{3} \}}$ can be build from two copies of the graph in \cref{fig:53-33_A}. It consists of $240$ vertices so that if $1 \leq i,j \leq 120 $, then there is edge between $i$ and $j$ if and only if there is an edge between $i+120$ and $j+120$. 
If $i,j$ are connected by an edge colour $3$ marked with a cross of colour $4$, then there is an edge colour $4$ connecting $i$ and $j'=j+120$ and an edge of colour $4$ connecting $j$ and $i'=i+120$ (see \cref{fig:5-33_B}).

From the graph described above and $t \geq 1$ we build the graph $\cX_{\{ 5, 3, {}_{3}^{3} \}}^{t}$ with vertex set $\left\{ 1, \dots, 240 \right\} \times \bZ_{t} $. 
The edges of $\cX_{\{ 5, 3, {}_{3}^{3} \}}^{t}$ are of the form $\{(i, \ell),  (j.\ell)\}$ (of colour $c$) whenever $i$ and $j$ are connected by an edge of colour $c$ in the graph $X_{\{ 5, 3, {}_{3}^{3} \}}$.   
For every $\ell \in \bZ_{t}$, there are two (dotted) edges of colour $4$ connecting the pairs of vertices $\{(240, \ell) , (118, \ell+1)\}$ and $\{(238, \ell), (120, \ell + 1)\}$.

For a given $\ell \in \bZ_{t}$, consider the set of vertices $A_{\ell}=\left\{ (i,\ell) : 1 \leq i \leq 70 \right\} $, $B_{\ell}=\left\{ (i,\ell) : 71 \leq i \leq 118 \right\} $, $A'_{\ell}=\left\{ (i,\ell) : 121 \leq i \leq 190 \right\} $ and $B'_{\ell}=\left\{ (i,\ell) : 191 \leq i \leq 238 \right\} $. 
Note that all the edges of colour $3$ induced by the edges marked with a  cross in \cref{fig:53-33_A} connect vertices in $A_{\ell}$ with vertices in $B_{\ell}$ and vertices in $A'_{\ell}$ with vertices in $B'_{\ell}$ whereas the edges of colour $4$ induced by the crosses connect vertices in $A_{\ell}$ with vertices in $B'_{\ell}$ and vertices in $A'_{\ell}$ with vertices in $B_{\ell}$ (see \cref{fig:53-33_B}).

\begin{figure}
\def\svgwidth{.6\textwidth}
\begingroup \makeatletter \providecommand\color[2][]{\errmessage{(Inkscape) Color is used for the text in Inkscape, but the package 'color.sty' is not loaded}\renewcommand\color[2][]{}}\providecommand\transparent[1]{\errmessage{(Inkscape) Transparency is used (non-zero) for the text in Inkscape, but the package 'transparent.sty' is not loaded}\renewcommand\transparent[1]{}}\providecommand\rotatebox[2]{#2}\newcommand*\fsize{\dimexpr\f@size pt\relax}\newcommand*\lineheight[1]{\fontsize{\fsize}{#1\fsize}\selectfont}\ifx\svgwidth\undefined \setlength{\unitlength}{612.47363281bp}\ifx\svgscale\undefined \relax \else \setlength{\unitlength}{\unitlength * \real{\svgscale}}\fi \else \setlength{\unitlength}{\svgwidth}\fi \global\let\svgwidth\undefined \global\let\svgscale\undefined \makeatother \begin{picture}(1,1.23760788)\lineheight{1}\setlength\tabcolsep{0pt}\put(0,0){\includegraphics[width=\unitlength,page=1]{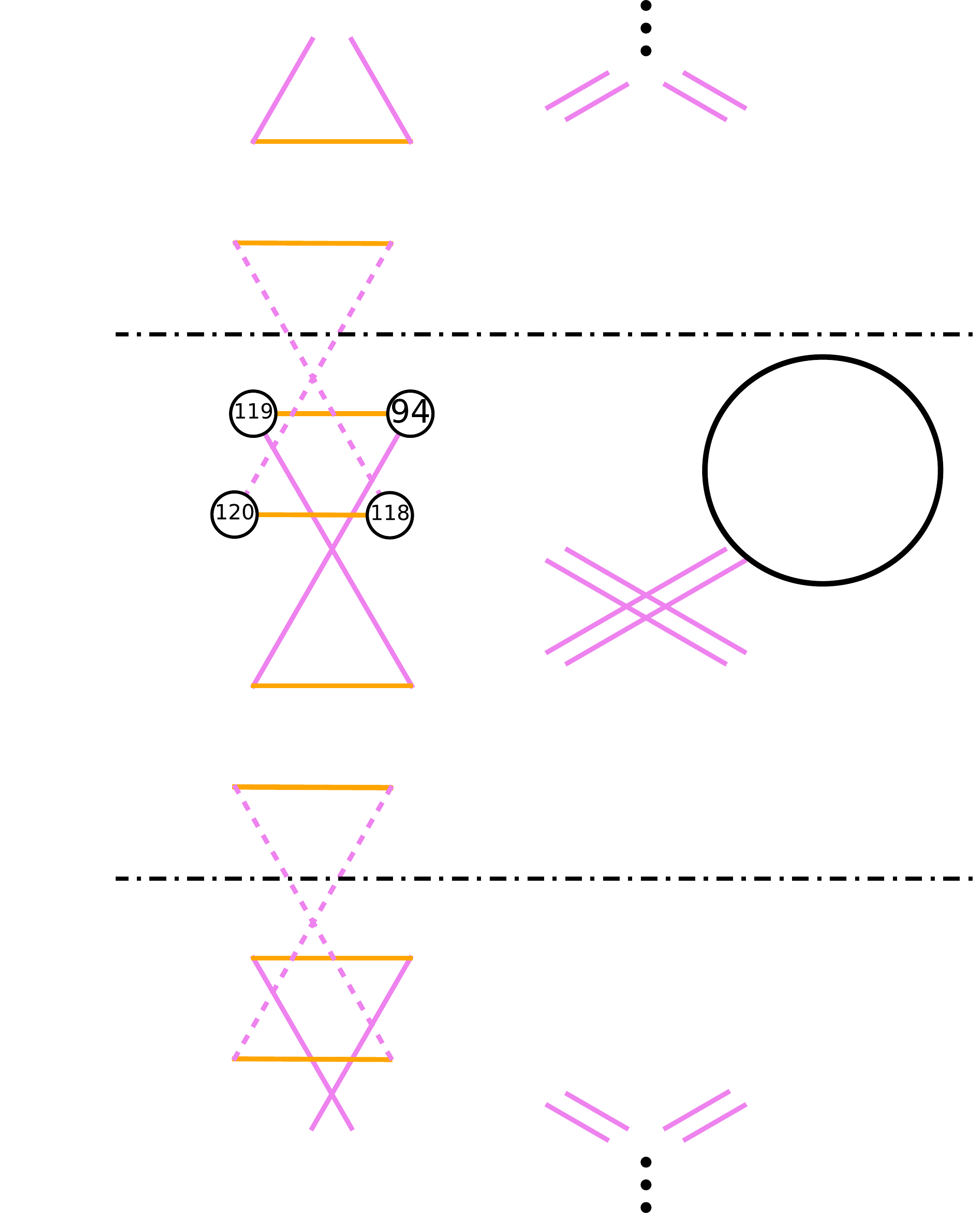}}\put(0.47894287,0.73450871){\makebox(0,0)[lt]{\lineheight{1.25}\smash{\begin{tabular}[t]{l}$B_{\ell}$\end{tabular}}}}\put(0.83967469,0.73450871){\makebox(0,0)[t]{\lineheight{1.25}\smash{\begin{tabular}[t]{c}$A_{\ell}$\end{tabular}}}}\put(0,0){\includegraphics[width=\unitlength,page=2]{53-33_B.pdf}}\put(0.47894287,0.45681719){\makebox(0,0)[lt]{\lineheight{1.25}\smash{\begin{tabular}[t]{l}$B'_{\ell}$\end{tabular}}}}\put(0.83967469,0.45681719){\makebox(0,0)[t]{\lineheight{1.25}\smash{\begin{tabular}[t]{c}$A'_{\ell}$\end{tabular}}}}\put(0,0){\includegraphics[width=\unitlength,page=3]{53-33_B.pdf}}\put(0.47894287,1.01220028){\makebox(0,0)[lt]{\lineheight{1.25}\smash{\begin{tabular}[t]{l}$B'_{\ell-1}$\end{tabular}}}}\put(0.83967469,1.01220028){\makebox(0,0)[t]{\lineheight{1.25}\smash{\begin{tabular}[t]{c}$A'_{\ell-1}$\end{tabular}}}}\put(0,0){\includegraphics[width=\unitlength,page=4]{53-33_B.pdf}}\put(0.47894287,0.17912562){\makebox(0,0)[lt]{\lineheight{1.25}\smash{\begin{tabular}[t]{l}$B_{\ell+1}$\end{tabular}}}}\put(0.83967469,0.17912562){\makebox(0,0)[t]{\lineheight{1.25}\smash{\begin{tabular}[t]{c}$A_{\ell+1}$\end{tabular}}}}\put(0,0){\includegraphics[width=\unitlength,page=5]{53-33_B.pdf}}\put(-0.00203288,1.03534125){\makebox(0,0)[lt]{\lineheight{1.25}\smash{\begin{tabular}[t]{l}$X_{\{ 5, 3, {}_{3}^{3} \}}^{\ell-1}$\end{tabular}}}}\put(-0.00203288,0.61880397){\makebox(0,0)[lt]{\lineheight{1.25}\smash{\begin{tabular}[t]{l}$X_{\{ 5, 3, {}_{3}^{3} \}}^{\ell}$\end{tabular}}}}\put(-0.00203288,0.20226661){\makebox(0,0)[lt]{\lineheight{1.25}\smash{\begin{tabular}[t]{l}$X_{\{ 5, 3, {}_{3}^{3} \}}^{\ell+1}$\end{tabular}}}}\put(0,0){\includegraphics[width=\unitlength,page=6]{53-33_B.pdf}}\end{picture}\endgroup  \caption{The graph $\cX_{\{ 5, 3, {}_{3}^{3} \}}^{t}$.  }
\label{fig:53-33_B}
\end{figure}

Consider the involution $\varphi$ of the vertices of $\cX_{\{ 5, 3, {}_{3}^{3} \}}^{t}$ defined by 
\[\begin{aligned} 
	(i,\ell) &\leftrightarrow (i+120,\ell) && \text{for } 1 \leq i \leq 70,\\
	(119,\ell) &\leftrightarrow (239,\ell), \\
	(240,\ell) &\leftrightarrow (120,\ell+1), 
\end{aligned}\] 
while fixing all the other vertices. 
For every $\ell \in \bZ_{t}$, the involution $\varphi$ swaps the sets $A_{\ell}$ and $A'_{\ell}$ and two pairs of vertices incident to the dotted lines. 
This induces a graph automorphism swapping the colours $3$ and $4$ while fixing the other three. 

The graph automorphism $\varphi$ induces a group-automorphism $\bar{\varphi}$ of the permutation group $G$ induced by the graph. 
The automorphism $\bar{\varphi}$ swaps the generators $\rho_{3}$ and $\rho_{4}$ while fixing $\rho_{i}$ for $i \leq 2$.
As mentioned before, the connected components induced by removing the edges of colour $4$ are isomorphic to the CPR-graph induced by  the action of the Coxeter group $[5,3,3]$ on the cells of a $120$-cell. 
This implies that the group $G_{4}$ is isomorphic to $[5,3,3]$.
The group automorphism $\bar{\varphi}$ maps the subgroup $G_{4}$ to the subgroup $G_{3}$, implying that the latter is also isomorphic to $[5,3,3]$.
Observe also that all the connect components of the graph induced by the colours $3$ and $4$ are isolated vertices or alternating cycles with $2$ or $4$ vertices.
The discussion above together with \cref{lem:polygonalAction} imply that the permutation group $G$ satisfies the relations implied by the Coxeter diagram in \cref{eq:cox53-33}.
Moreover, the automorphism $\bar{\varphi}$ can be seen as a horizontal symmetry of the diagram.

\begin{equation}\label{eq:cox53-33}
	\begin{tikzcd}[column sep=1em, row sep=.3em, ]
 	& & & [-0.3em] \overset{\rho_{3}}{\textcolor{Orange}{\bullet}} \arrow[dddl, dash] \\
 	& & & \\
 	& & & \\
 	\underset{\rho_{0}} {\textcolor{Red}{\bullet}} \arrow[r, dash, "5"]
 	&\underset{\rho_{1}} {\textcolor{Green}{\bullet}} \arrow[r, dash,]
 	&\underset{\rho_{2}}{\textcolor{Blue}\bullet} \arrow[dr, dash] & \\
	& & & [-0.3em] \underset{\rho_{4}}{\textcolor{Violet}{\bullet}} 
\end{tikzcd}
\end{equation}

Observe that the $\left\{1,2,3,4 \right\}$-component of the vertex $(120,0)$ consists of $8$ vertices.
Moreover, this is the CPR-graph induced by the Coxeter group $[3, {}_{3}^{3}]$ on the vertices of the $4$-dimensional cross-polytope (recall that the group $[3, {}_{3}^{3}]$ can be seen as an index $2$ subgroup of $[3,3,4]$). 
This implies that the subgroup $G_{0}$ is isomorphic to $[3, {}_{3}^{3}]$. 
The group $G_{1} $ is the direct product $ \left\langle \rho_{0} \right\rangle \times \left\langle \rho_{3}, \rho_{2}, \rho_{4} \right\rangle \cong  [2,3,3]$.
Similarly, $G_{2}=\left\langle \rho_{0}, \rho_{1} \right\rangle \times \left\langle \rho_{3} \right\rangle \times \left\langle \rho_{4} \right\rangle  \cong [5,2,2]$.
This implies that the subgroup $G_{i}$ is isomorphic to a finite Coxeter group for every $i \in \left\{ 0,1,2,3,4 \right\} $, hence, it satisfies the intersection property.

To show that the group $G$ satisfies the intersection property we just need to prove that $G_{i}\cap G_{j} = G_{i,j}$ for every  $\left\{ i,j \right\} \subset \left\{ 0,1,2,3,4 \right\} $.
To do so we use \cref{lem:IP_CPR}. 
More precisely, we list below a vertex $x$ that satisfies \cref{eq:IP_gcd} for some of the pairs $\left\{ i,j \right\} $.
For the remaining pairs we prove that there exists a vertex $x$ that satisfies \cref{eq:IP_CPR}.
\begin{itemize}
	\item If $\left\{ i,j \right\} =\left\{ 0,1 \right\}  $ take $x=(1,0)$.
	\item If $\left\{ i,j \right\} =\left\{ 0,2 \right\}  $ take $x=(1,0)$.
	\item If $\left\{ i,j \right\} =\left\{ 1,2 \right\}  $ take $x=(1,0)$.
	\item If $\left\{ i,j \right\} =\left\{ 2,3 \right\}  $ take $x=(3,0)$.
	\item If $\left\{ i,j \right\} =\left\{ 0,4 \right\}  $ take $x=(3,0)$.
	\item If $\left\{ i,j \right\} =\left\{ 3,4 \right\}  $ take $x=(120,0)$.
\end{itemize}

For the pair $\left\{ 0,3 \right\} $ take $x=(117,0)$.
Notice that $\left| xG_{0} \cap xG_{3} \right| = 4$.
Observe that $\stab_{G_{0}}(x)=\left\langle \rho_{2}, \rho_{3},\rho_{4} \right\rangle \cong [3,3]$.
Note also that \[\left\langle \rho_{2},\rho_{4} \right\rangle \leq \stab_{G_{0}}(x) \cap \stab_{G_{3}} \leq \stab_{G_{0}}(x).\]
However, $\left\langle \rho_{1}, \rho_{2} \right\rangle $ is maximal in $\left\langle \rho_{1}, \rho_{2}, \rho_{4} \right\rangle $, which implies that $\left| \stab_{G_{0}} \cap \stab_{G_{3}} \right| = 6$, hence \[\left| xG_{0} \cap xG_{3} \right|\cdot \left| \stab_{G_{0}} \cap \stab_{G_{3}} \right| = 24 = G_{0,3}. \]
A similar argument holds for the pair $\left\{ 0,4 \right\} $.

Likewise, for the pair $\left\{ 1,3 \right\} $ take the vertex $x=(118,0)$. Note that $\left| xG_{1}\cap xG_{3} \right| = 3 $ and that $\stab_{G_{1}}(x)=\left\langle \rho_{0}, \rho_{3}, \rho_{4} \right\rangle $.
\[\left\langle \rho_{0},\rho_{4} \right\rangle \leq \stab_{G_{1}}(x) \cap \stab_{G_{3}}(x) \lneq \left\langle \rho_{0},\rho_{3},\rho_{4} \right\rangle. \]
The group $\left\langle \rho_{0},\rho_{4} \right\rangle$ is of index $2$, hence maximal in $\left\langle \rho_{0},\rho_{3},\rho_{4} \right\rangle$. This implies that $\left\langle \rho_{0},\rho_{4} \right\rangle = \stab_{G_{1}}(x) \cap \stab_{G_{3}}(x)$ and that $x$ satisfies \cref{eq:IP_CPR}.
The exact same argument can be used for the pair $\left\{ 1,4 \right\} $.

The intersection property for $G$ follows from \cref{prop:IP}.

By \cref{thm:FT_rank3}, to show that the group $G$ is flag-transitive we only need to prove that \[G_{i}\cap G_{j}G_{k} = G_{i,j} G_{i,k}\] for every  $\left\{ i,j,k \right\} \subset \left\{ 0,1,2,3,4 \right\}$.

In the list below we give a vertex $x$ for some of the ordered triples $(i,j,k)$ so that \cref{eq:FT_CPR} holds for such vertex.
\cref{lem:FT_CPR} implies that $G_{i}\cap G_{j}G_{k} = G_{i,j} G_{i,k}$ for the given subset $\left\{ i,j,k \right\} $.

\begin{itemize}
	\item $x=(7,0)$ for $(i,j,k)=(0,1,2)$.
	\item $x=(7,0)$ for $(i,j,k)=(0,1,3)$.
	\item $x=(7,0)$ for $(i,j,k)=(0,1,4)$.
	\item $x=(1,0)$ for $(i,j,k)=(2,0,3)$.
	\item $x=(1,0)$ for $(i,j,k)=(2,0,4)$.
	\item $x=(2,0)$ for $(i,j,k)=(1,2,3)$.
	\item $x=(2,0)$ for $(i,j,k)=(1,2,4)$.
	\item $x=(1,0)$ for $(i,j,k)=(2,3,4)$.
\end{itemize}

For the subset $\left\{ 0,3,4 \right\} $ it can be shown that \[G_{0} \cap \left( \bigcup_{\alpha \in R}\alpha G_{4} \right) \bigcup_{\alpha \in R}\left( G_{0}\cap \alpha G_{4} \right)  = (G_{0,3})(G_{0,4}),\] where $R$ is a set of coset representatives of $G_{3,4}$ in $G_{3}$.
We do not give the explicit computations because they are long but straightforward, but we briefly explain how to prove it.

First observe that $G_{_{3,4}}$ has index $120$ in $G_{3}$, which implies that $R$ must have $120$ elements.
Since $G_{3}$ is a Coxeter group, the set $R$ is easy to compute (see \cite[Section 1.10]{Humphreys_1990_ReflectionGroupsCoxeter}).
Moreover, $R$ can be computed so that $R_{0} := \left\{ \id, \rho_{4}, \rho_{2}\rho_{4}, \rho_{1}\rho_{2}\rho_{4} \right\} \subset R $.
Observe that $R_{0} \subset G_{0}$, which implies that 
\[\begin{aligned} 
\left| \bigcup_{\alpha \in R_{0}} \left( G_{0} \cap \alpha G_{4} \right) \right|  
&= \left| \bigcup_{\alpha \in R_{0}} \alpha\left( G_{0} \cap G_{4} \right) \right|\\ 
&= 4(\left| G_{0,4} \right| ) \\
&= \left( \frac{|G_{0,3}|}{|G_{0,3,4}|} \right)\left( |G_{0,4}| \right)
&= |(G_{0,3})(G_{0,4})|.
\end{aligned}\] 

Therefore, it remains to show that $G_{0}\cap \alpha G_{4} =\emptyset$ for $\alpha \in R \sm R_{0}$.
By a similar argument to the one used previously, observe that if $\beta \in G_{0}$, then $G_{0}\cap \alpha G_{4} =\emptyset$ if and only if $G_{0}\cap \beta \alpha G_{4} =\emptyset$.
This observation reduces the condition above to a set $R' \subset R$ with the property that every element in $R$ is either in $R'$ of is of the form $\beta \alpha$ for some $\beta \in G_{0}$ and $\alpha \in R'$.
The set $R'$ consists of $21$ elements and for each $\alpha \in R'$ we can find a vertex $x$ such that $xG_{0} \cap (x \alpha) G_{4} = \emptyset$, implying that $G_{0} \cap \alpha G_{4} = \emptyset$ for each $\alpha \in R'$.
This proves that $G_{0} \cap G_{3}G_{4} = (G_{0,3})(G_{0,4})$.

A very similar approach can be used for the tuple $(1,3,4)$ to show that \[G_{1} \cap G_{3} G_{4} = (G_{1,3})(G_{1,4}).\]

Flag transitivity for the group $G$ follows from \cref{thm:FT_rank3}.

\section{Conclusions}

The constructions in 
\cref{sec:constructions} give us an infinite family of locally spherical regular hypertopes for each hyperbolic type. 
More precisely, for each hyperbolic type $D$ and each positive integer $t$, we build a properly-edge-coloured graph $\cX_{D}^{t}$ and prove that for all but a few integers $t$, the induced permutation group is the type-preserving automorphism group of a regular hypertope $\cH^{t}_{D}$ of type $D$. 
It should be pointed out that if $s \neq t$ it does not follow that $\cH^{s}_{D} \not\cong \cH^{t}_{D}$. 
For example, it can be easily checked that the hypertopes $\cH^{1}_{\left( 3,3,3,4 \right)}$ and $\cH^{2}_{\left( 3,3,3,4 \right)}$ are isomorphic.
The previous observation does not necessarily imply that our constructions do not yield an infinite family of regular hypertopes for each hyperbolic type, but it makes it less obvious. 
The following result justifies our claim.

\begin{proposition}
	Let $D$ be a diagram of hyperbolic type and $t \geq 1$ an integer. Assume that $\cH^{t}_{D}$ is the regular hypertope built in \cref{sec:constructions} for the corresponding $D$ and $t$, then $t$ divides the order of $\aut_{I}(\cH^{t}_{D})$.
\end{proposition}
\begin{proof}
	Let $(y,-1)$ and $(x,0)$ be two vertices in $\cX_{D}^{t}$, the CPR-graph associated with $\cH^{t}_{D}$, such that $\left\{ (y,0), (x,0) \right\} $  is a dotted edge of colour $i$.
	
	Let $X^{0}_{D}$ be the subgraph of $\cX_{D}^{t}$ consisting of the vertices whose second coordinate is $0$, and the solid edges connecting two of them.
	Since $X^{0}_{D}$ is connected, there exist a path $P$ (of solid edges) connecting $(x,0)$ and $(y,0)$.
	The path $P$ induces a group element $\omega$. 
	Finally observe that $(x, \ell)\omega\rho_{i} = (x, \ell+1)$ for every $\ell \in \bZ_{t}$.
	This implies that the orbit of $(x,0)$ under the cyclic group $\left\langle \omega\rho_{i} \right\rangle $ has length $t$ and therefore $t$ is a divisor of $\left| \left\langle \omega\rho_{i} \right\rangle  \right| $, which in turn divides $\left| \aut_{I}\left( \cH^{t}_{D} \right) \right| $.
	
\end{proof}

\begin{corollary}
	For every positive integer $t$ and every diagram of hyperbolic type $D$ there exists a regular hypertope $\cH$ with Coxeter diagram $D$ such that \[\left| \aut_{I}\left( \cH \right) \right| \geq t. \]
\end{corollary}

\section*{Acknowledgements}

Both authors were supported by the Natural Sciences and Engineering Research Council of Canada (NSERC Canada). 
The first author was also supported by the Post Doctoral Scholarship Program at UNAM, Mexico.
The research on this paper was partially developed while the first author was a Postdoctoral Visitor in the Department of Mathematics and Statistics, York University, Canada.

\printbibliography

\end{document}